\documentclass[a4paper]{siamart1116}
\usepackage[english]{babel}

\newsiamthm{conjecture}{Conjecture}
\newsiamremark{example}{Example}
\newsiamremark{remark}{Remark}

\usepackage[utf8]{inputenc}

\usepackage{tkz-euclide, tikz}
\usetikzlibrary{backgrounds,patterns,matrix,calc,arrows,positioning,fit,shapes.geometric,decorations.pathmorphing}
\usepackage{float}
\usepackage{amsmath,amssymb,amsfonts}
\usepackage{multirow}
\usepackage{enumitem}
\usepackage{tabularx}
\usepackage{booktabs}
\usepackage{graphicx}
\usepackage[all,2cell]{xy}
\usepackage[sort]{cite}
\usepackage{textcomp}

\newcommand{\bb}[1]{\mathbb{#1}}

\newcommand{\calZ}[1]{\mathcal{#1}}

\newcommand{\set}[1]{\left\{#1\right\}}
\newcommand{\cset}[2]{\left\{#1\mid #2\right\}}

\newcommand{\Norm}[1]{\lVert #1 \rVert}
\newcommand{\rank}{\mathrm{rk}\,}

\newcommand{\tuple}[1]{\mathfrak{#1}}

\newcommand{\Var}[1]{\mathcal{#1}}
\newcommand{\Sec}[2]{\sigma_{#1}(#2)}

\newcommand{\tensor}[1]{\mathfrak{#1}}
\newcommand{\vect}[1]{\mathbf{#1}}
\newcommand{\sten}[3]{\vect{#1}_{#2}^{#3}}
\newcommand{\Tang}[2]{\mathrm{T}_{#1} {#2}}

\newcommand{\R}{\mathbb{R}}
\newcommand{\deriv}[2]{\mathrm{d}_{#2}#1}

\newcommand{\ldbracket}{\text{\textlbrackdbl}}
\newcommand{\rdbracket}{\text{\textrbrackdbl}}

\newcommand{\Gr}{\mathrm{Gr}}

\newcommand{\GrSigma}{\Sigma_{\Gr}}

\newcommand{\dist}{\mathrm{dist}}
\newcommand{\JOIN}{\mathrm{Join}}

\newcommand{\Pj}{\mathbb{P}}

\newcommand{\refthm}[1]{{\cref{#1}}}
\newcommand{\reflem}[1]{{\cref{#1}}}
\newcommand{\refeqn}[1]{{(\ref{#1})}}

\newcommand{\refsec}[1]{{\cref{#1}}}

\newcommand{\refcor}[1]{{\cref{#1}}}
\newcommand{\reffig}[1]{{\cref{#1}}}
\newcommand{\reftab}[1]{{\cref{#1}}}
\newcommand{\refprop}[1]{{\cref{#1}}}

\newcommand{\refrem}[1]{{Remark \ref{#1}}}

\numberwithin{equation}{section}
\numberwithin{figure}{section}
\numberwithin{table}{section}
\numberwithin{theorem}{section}

\title{The condition number of join decompositions\thanks{Submitted to \funding{The first author was partially supported by DFG research grant BU 1371/2-2. The second author was supported by a Postdoctoral Fellowship of the Research Foundation--Flanders (FWO).}}
}
\author{
Paul Breiding\thanks{Max-Planck Institute for Mathematics in the Sciences Leipzig. (\email{breiding@mis.mpg.de})}
\and
Nick Vannieuwenhoven\thanks{KU Leuven, Department of Computer Science. (\email{nick.vannieuwenhoven@cs.kuleuven.be})}}

\allowdisplaybreaks
\begin{document}

\maketitle

\begin{abstract}
The join set of a finite collection of smooth embedded submanifolds of a mutual vector space is defined as their Minkowski sum. Join decompositions generalize some ubiquitous decompositions in multilinear algebra, namely tensor rank, Waring, partially symmetric rank and block term decompositions.
This paper examines the numerical sensitivity of join decompositions to perturbations; specifically, we consider the condition number for general join decompositions.
It is characterized as a distance to a set of ill-posed points in a supplementary product of Grassmannians. We prove that this condition number can be computed efficiently as the smallest singular value of an auxiliary matrix.
For some special join sets, we characterized the behavior of sequences in the join set converging to the latter's boundary points. Finally, we specialize our discussion to the tensor rank and Waring decompositions and provide several numerical experiments confirming the key results.
\end{abstract}





\section{Introduction}\label{sec:introduction}
Many subsets of $\R^N$ admit the structure of a \emph{join set}:
the join $\Var{J} = \JOIN(\Var{M}_1,\ldots,\Var{M}_r)$ of a collection of smooth embedded submanifolds $\Var{M}_i \subset \R^N$ is defined as the image of the addition map
\begin{equation}\label{eqn:phi}
\Phi:\calZ{M}_1\times\ldots \times \calZ{M}_r \longrightarrow \bb{R}^N,\quad
(\vect{p}_1,\ldots,\vect{p}_r)\longmapsto \vect{p}_1 + \cdots + \vect{p}_r.
\end{equation}
The class of join sets comprises, among others, tensor canonical rank decompositions, Waring decompositions, partially symmetric rank decompositions, and block term decompositions.
For example, the set of $m\times n$ matrices of rank equal to $1$ is a manifold that we refer to as the Segre manifold $\Var{S}_{m,n}$. Constructing the join set $\Sec{r}{\Var{S}_{m,n}} := \JOIN(\Var{S}_{m,n},\ldots,\Var{S}_{m,n})$, we obtain exactly the algebraic variety of matrices of rank at most $r$ \cite{Landsberg2012}. This example generalizes to higher-order tensors, as the set of rank-$1$ tensors in $\R^{m_1} \otimes \cdots \otimes \R^{m_d} \simeq \R^{m_1 \cdots m_d}$ forms the Segre manifold $\Var{S} := \Var{S}_{m_1,\ldots,m_d}$ \cite{Landsberg2012}. The join set $\Sec{r}{\Var{S}} := \operatorname{Join}(\Var{S},\ldots,\Var{S})$ is then the semi-algebraic set consisting of the tensors of rank at most $r$ \cite{dSL2008}.
Note that these examples are special because $\Var{M}_1 = \cdots = \Var{M}_r$, in which case we refer to the join set as an \emph{$r$-secant} set. Block term decompositions \cite{Lathauwer2008} where the summands have different multilinear ranks are an example of a join set that is not a secant set. An example of secant sets not involving tensors is given in algebraic vision; if the focal locus of a certain camera model, called \emph{congruence}, is a curve $Y$, then the congruence is given by the set of secants $\Sec{2}{Y}$ \cite{congruences,congruences2}.

The \emph{join decomposition problem} (JDP) is a natural computational problem associated with a join set, comprising several well-known tensor decomposition problems as particular instances.
Given smooth embedded manifolds $\Var{M}_1,\ldots,\Var{M}_r \subset \R^N$ and a $\vect{p}\in\JOIN(\Var{M}_1,\ldots,\Var{M}_r )$ the JDP is this:
\begin{equation*}
\begin{array}{l}
\text{Compute $\tuple{p}\in\Var{M}_1\times \ldots \times \Var{M}_r$ such that $\Phi(\tuple{p})=\vect{p}$.}
\end{array}
\tag{JDP}\label{eqn_JDP}
\end{equation*}
{From our applications in tensor decompositions,} we are particularly interested in the case where the JDP has only a finite number of solutions for a given $\vect{p} \in \Var{J}$.
In fact, most of the aforementioned join sets even offer some guarantees of a \emph{unique} solution of the JDP, usually up to a permutation of the summands; this is the case for tensor rank decompositions \cite{CO2012,COV2014,COV2016,DDL2013Part2,Kruskal1977,JS2004}, Waring decompositions \cite{COV2017,COV2016}, partially symmetric decompositions \cite{Kruskal1977} and specific types of block term decompositions \cite{Lathauwer2008,Yang2013}.
These uniqueness properties are often of significant practical value, e.g., for identifying the parameters of certain latent variable models \cite{AMR2009,AGHKT2014}.

While it happens for theoretical reasons that $\vect{p} \in \R^N$ belongs exactly to $\Var{J}$, this is an uncommon situation in applications. The reason is that $\vect{p}$ is often corrupted by different sources of noise; for example, representation errors due to roundoff are typically incurred when representing $\vect{p}$ on a computer. Further sources of error could be measurement errors, numerical approximation errors, accumulation of roundoff errors, and modeling errors. For these reasons, a more common computational problem is the \emph{join approximation problem} (JAP), which consists of finding a point on a join set that is a closest approximation to a given point $\widetilde{\vect{p}} \in \R^N$:
\begin{align}\label{eqn_prototype_optim_problem}
\tag{JAP} \text{Compute $\widetilde{\tuple{p}}\in\Var{M}_1\times \ldots \times \Var{M}_r$ such that $\widetilde{\tuple{p}} \in \underset{\tuple{q} \in \Var{M}_1 \times \cdots \times \Var{M}_r}{\arg\min} \|\Phi(\tuple{q}) - \widetilde{\vect{p}}\|$.}
\end{align}
For now we ignore the question whether the above problem is actually well-posed; see \refsec{sec_boundary_points} for our treatment of this issue.

The key question addressed in this paper is the following. If $\tuple{p}$ is the decomposition of the true but unobservable model $\vect{p} = \Phi(\tuple{p})$ and $\widetilde{\tuple{p}}$ is the decomposition obtained from solving (approximately) the JAP with an observed $\widetilde{\vect{p}} \approx \vect{p}$ as input, then what is the relationship between the decompositions $\tuple{p}$ and $\widetilde{\tuple{p}}$?
We focus on the first-order behavior of this relationship which is captured by the \emph{condition number} of the JDP.

Recall that an \emph{absolute condition number} of a function $f : D \to I$ at $x \in D$ is classically defined in numerical analysis \cite{Wilkinson1963,Higham1996} as the maximum magnification of an infinitesimal input perturbation by the function:
\begin{equation}\label{classical_condition}
\kappa[f](x) := \lim_{\epsilon \to 0} \; \max_{x' \in (B_{\epsilon}(x) \cap D)} \frac{\| f(x) - f(x') \|_I }{\| x - x' \|_D},
\end{equation}
where $\|\cdot\|_D$ and $\|\cdot\|_I$ are norms on the domain $D$ and image $I$ respectively, and $B_{\epsilon}(x)$ is the $\epsilon$-ball about $x$ in the norm $\|\cdot\|_D$.
In this paper, we thus seek the condition number of $\Phi^{-1}$. A complication immediately arises because a join decomposition is usually not strictly unique; that is, $\Phi$ is not injective. This entails that the condition number would always be $\infty$ by the above definition. Nevertheless, $\Phi$ can still be \emph{locally invertible}; that is, there exists an open neighborhood $\Var{N} \subset \Var{M}_1 \times \cdots \times \Var{M}_r$ of~$\tuple{p}$ and a \emph{local inverse function} $\Phi_{\tuple{p}}^{-1}$ at $\tuple{p}$ so that $\Phi_{\tuple{p}}^{-1} \circ \Phi|_{\Var{N}} = \operatorname{id}_{\Var{N}}$. In this case it is still possible to investigate the \emph{local condition number} of the JDP at $\tuple{p}$:
\begin{equation}\label{eqn_local_cond}
\kappa (\tuple{p}, \vect{p}) :=
\begin{cases}
\kappa[\Phi_{\tuple{p}}^{-1}](\Phi(\tuple{p})) & \text{if there is a local inverse function } \Phi_{\tuple{p}}^{-1}, \\
\infty& \text {otherwise.}
\end{cases}
\end{equation}
where $\vect{p}=\Phi(\tuple{p})$ and $\Phi_{\tuple{p}}^{-1}(\vect{p}) = \tuple{p}$. Because $\vect{p}$ is implicitly given by $\tuple{p}$, we omit $\vect{p}$ in the notation \refeqn{eqn_local_cond} and in the following we simply write $\kappa(\tuple{p})$ instead.

This local condition number captures $\tuple{p}$'s sensitivity to structured perturbations\footnote{By a structured perturbation of $\vect{p}$ we mean that the perturbation is such that $\widetilde{\vect{p}}$ is still a point in the join.} of $\vect{p} = \Phi(\tuple{p})$. In applications where the components of $\tuple{p} = (p_1,\ldots,p_r)$ have an interpretation, this condition number is naturally of interest because it bounds the perturbations to the $p_i$'s in the given join decomposition $\tuple{p}$. Moreover, \refeqn{eqn_local_cond} obeys the infamous rule of thumb of classic numerical analysis:
\begin{align} \label{eqn_rule_of_thumb}
 \underset{\text{forward error}}{\underbrace{\| \tuple{p} - \Phi_{\tuple{p}}^{-1}(\vect{w}) \|}} \lesssim \underset{\text{condition number}}{\underbrace{\kappa(\tuple{p})}} \cdot \underset{\text{backward error}}{\underbrace{\| \Phi(\tuple{p}) - \vect{w} \|}}
\end{align}
for all $\vect{w}$ in a small neighborhood of $\Phi(\tuple{p})$. Recall that this rule is one of the main applications of condition numbers in data analysis applications where one usually seeks to estimate the forward error based only on computable information \cite{Higham1996,Wilkinson1963}.

Following \cite[Remark 14.14]{BC2013}, the \emph{maximum local condition number} of the JPD at $\vect{p}$ is naturally defined as the maximum of the local condition numbers over the fiber of $\Phi$ at $\vect{p}$: $K(\vect p):= \max_{\tuple p \in \Phi^{-1}(\vect{p})} \kappa(\tuple p).$ Obtaining meaningful results for this condition number is vastly more complicated as we believe it requires deep understanding of the global geometry of join sets. This lies beyond the scope of this work.
Note that in the case of the aforementioned tensor decompositions, which are often unique up to the order of the summands, the condition number \refeqn{eqn_local_cond} is the same for all decompositions and hence coincides with the maximum condition number.

Despite the importance of condition numbers in numerical analysis and the interest in specific types of join decompositions, heretofore few results exist on their conditioning. Only recently a local condition number was proposed for the tensor rank decomposition \cite{V2017}.
In this paper, we aspire towards treating JDPs in significant generality using a more elegant approach.
Another notion of the stability of the parameters of a tensor rank decomposition is the Cram\'er--Rao bound; its connection to the condition number of the factor matrices is discussed in \cite{V2017}.

\subsection{Contributions}
For brevity, we call $\kappa(\tuple{p})$ in \refeqn{eqn_local_cond} the condition number in the remainder of the paper, its locality being implicit.
We derived the condition number of the JDP via a local application of the framework of B\"urgisser and Cucker \cite{BC2013}. In the rest of this subsection, let $\Var{M}_i$ be an $n_i$-dimensional embedded submanifold of~$\R^N$, and $n:=n_1+\cdots+n_r$. Moreover, let $\Tang{x}{\Var{M}}$ denote the tangent space to a manifold~$\Var{M}$ at $x \in \Var{M}$.

\begin{theorem}[Spectral characterization]
\label{thm_condition_number}
Let $\varsigma_n(A)$ denote the $n$th largest singular value of the linear operator $A$. The condition number of the JDP at the decomposition $\tuple{p} = (p_1,\ldots,p_r) \in \Var{M}_1\times\cdots\times\Var{M}_r$ is
\begin{align} \label{eqn_main_thm_condition}
 \kappa(\tuple{p}) = \frac{1}{\varsigma_{n}( \deriv{\Phi}{\tuple{p}} )} = \frac{1}{\varsigma_n(U)},
\end{align}
where $\deriv{\Phi}{\tuple{p}} : \Tang{\tuple{p}}{(\Var{M}_1 \times \cdots \times \Var{M}_r)} \to \R^N$ is the derivative of the addition map $\Phi$ at $\tuple{p}$, and $U := \begin{bmatrix} U_1 & \cdots & U_r \end{bmatrix}$ wherein $U_i \in \R^{N \times n_i}$ is an orthonormal basis of $\Tang{p_i}{\Var{M}_i}$.
\end{theorem}

Note that contrary to definition \refeqn{eqn_local_cond}, the expression derived above is efficiently computable, as it suffices computing the least singular value of a matrix.

An important consequence of \refthm{thm_condition_number} is that the condition number $\kappa(\tuple{p})$ is unbounded if the join decomposition of $\Phi(\tuple{p})$ is not locally unique on a smooth submanifold. Under these conditions, the derivative $\deriv{\Phi}{\tuple{p}}$ is not injective, by the contrapositive of \cite[Proposition 4.1]{Lee2013}, so that $\kappa(\tuple{p}) = \infty$ by \refthm{thm_condition_number}.

\begin{corollary} \label{prop_infdecomp_infcond}
If there exists a smooth submanifold $\Var{E} \subset \Var{M}_1 \times \cdots \times \Var{M}_r$ for which $\Phi(\tuple{p}) = \Phi(\tuple{q})$ for all $\tuple{p}, \tuple{q} \in \Var{E}$, then every $\tuple{p}$ in the interior of $\Var{E}$ has $\kappa(\tuple{p}) = \infty$.
\end{corollary}

Another key contribution of this paper is the characterization of the condition number $\kappa(\tuple{p})$ of JDPs at a decomposition $\tuple{p}$ as an inverse distance to a set of \emph{ill-posed} inputs in the spirit of Demmel \cite{Demmel1987}.

\begin{theorem}[Characterization as inverse distance] \label{thm_inverse_distance}
The condition number of the JDP at $\tuple{p}=(p_1,\ldots,p_r)\in \Var{M}_1\times\cdots\times \Var{M}_r$ is
\begin{align*}
\kappa(\tuple{p})
= \frac{1}{\dist\left( (\Tang{p_1}{\Var{M}_1}, \ldots, \Tang{p_r}{\Var{M}_r} ),\GrSigma\right)}
\end{align*}
where the distance measure is defined as in \refeqn{proj-distance}, and the \emph{ill-posed locus} $\Sigma_\Gr$ is
$$
\Sigma_\Gr= \{ (W_1,\ldots,W_r) \in \operatorname{Gr}(N,n_1) \times \cdots \times \operatorname{Gr}(N,n_r) \;|\; \dim (W_1+\cdots+W_r) < n \}
$$
with $\operatorname{Gr}(N,k)$ the Grassmannian of $k$-dimensional subspaces of $\R^N$.
\end{theorem}

The last main contribution concerns the behavior of the condition number near the boundary of a join set.
Indeed, many join sets $\Var{J} = \operatorname{Join}(\Var{M}_1,\ldots,\Var{M}_r)$ are not closed in the Euclidean topology; see \refrem{rem_open_boundary_exists}. Consequently, there may exist \emph{convergent} sequences $\vect{p}_0, \vect{p}_1, \ldots$ with every $\vect{p}_i \in \Var{J}$ but for which nevertheless $\lim_{i\to\infty} \vect{p}_i \to \vect{p}_\star \not\in \Var{J}$. In the context of tensor rank decompositions, de Silva and Lim \cite{dSL2008} explained what may be expected of such sequences: for $\tuple{p}_i = (p_1^{(i)}, \ldots, p_r^{(i)}) \in \Var{M}_1 \times \cdots \times \Var{M}_r$ with $\vect{p}_i = \Phi(\tuple{p}_i)$ some of the summands $p_j^{(i)}$ become of unbounded norm when $\vect{p}_i \to \vect{p}_\star$. In fact, this result generalizes to join decompositions. Moreover, we show that also \emph{the condition number of the $\tuple{p}_i$'s becomes unbounded} in this case. This provides a more rigorous criterion for determining if a sequence $\vect{p}_i \in \Var{J}$ is converging to a $\vect{p}_\star \not\in \Var{J}$.

\begin{theorem} \label{thm_boundary_points_cn}
Assume that
$
\JOIN(\Var{M}_1,\ldots,\Var{M}_r) = \JOIN(\overline{\Var{M}_1},\ldots,\overline{\Var{M}_r}),
$
where the overline denotes Euclidean closure, and that $\Var{M}_i$ is a cone.\footnote{We call a set $C\subseteq \R^N$ a cone, if $p\in C$ implies that $\alpha p\in C$ for all $\alpha\in\R\backslash\set{0}$.} Let $\vect{p}_i(t) \subset \Var{M}_i$, $1\leq i\leq r$, be a set of $r$ smooth curves simultaneously defined for all $t \in (0, 1)$.
If
\[
 \lim_{t \to 0}  (\vect{p}_1(t)+\ldots+ \vect{p}_r(t)) = \vect{p}_\star \;\not\in \JOIN(\Var{M}_1,\ldots,\Var{M}_r)
\]
then the condition number grows without bound:
\[
\lim_{t\to0} \kappa( \vect{p}_1(t),\ldots,\vect{p}_r(t) ) \to \infty;
\]
moreover, at least two summands diverge:
\[
\exists 1 \le i_\star < j_\star \le r: \lim_{t\to0} \| \vect{p}_{i_\star}(t) \| \to \infty \text{ and } \lim_{t\to0} \|\vect{p}_{j_\star}(t)\| \to \infty.
\]
\end{theorem}

\subsection{Outline}
The rest of this paper is structured as follows.
In \refsec{sec_condition_number} we derive the condition number of the JDP and prove \refthm{thm_condition_number}. We characterize it as an inverse distance in \refsec{sec_def_condition_number}, establishing \refthm{thm_inverse_distance}. The condition number of boundary points of join sets is investigated in \refsec{sec_boundary_points}, and we prove \refthm{thm_boundary_points_cn}. In \refsec{sec_examples} the discussion is specialized to two important join decompositions, namely the Waring and tensor rank decompositions. Some numerical experiments involving tensor rank decomposition problems are presented in \refsec{sec_numerical_experiments}. The final section summarizes our main conclusions.

\subsection{Notation}
We fix some notation for the rest of this article. Vectors are typeset in lowercase boldface letters ($\vect{p}$), matrices in uppercase ($U$), tensors in an uppercase Fraktur font ($\tensor{A}$), tuples in a lowercase Fraktur font ($\tuple{p}$), and manifolds, varieties and join sets in an uppercase calligraphic font ($\Var{M}$).

Let $A \in \R^{m \times n}$. The transpose of $A$ is denoted by $A^T$, and $A$'s pseudo-inverse is denoted by $A^\dagger$. The $i$th column of $A$ is denoted by $\vect{a}_i$. We denote by $\varsigma_k(A)$ the $k$th largest singular value of $A$. Its spectral norm is $\Norm{A} := \varsigma_1(A)$. The rank of $A$ is $\rank(A)$. The $m \times m$ identity matrix is denoted by $I_m$.

We define $\R^m_0 := \R^m\setminus\set{0}$.
The standard inner product on $\bb{R}^n$ is denoted by $\langle \vect{x},\vect{y} \rangle =\vect{x}^T \vect{y}$, which induces the norm $\Norm{\vect{x}} =\sqrt{\langle \vect{x}, \vect{x}\rangle}$. Given $\vect{x}, \vect{y}\in\bb{R}^{n}_0$ the angle between $\vect{x}$ and $\vect{y}$ is $\sphericalangle(\vect{x},\vect{y}) := \arccos\frac{ \langle \vect{x},\vect{y}\rangle}{\Norm{\vect{x}}\Norm{\vect{y}}}$. The unit sphere in $\bb{R}^n$ is denoted by $\bb{S}(\bb{R}^n):=\cset{\vect{x}\in\bb{R}^n}{\Norm{\vect{x}}=1}$. The ball of radius~$r$ around $x \in X$ is $B_r(x) = \cset{y}{\|x-y\| \le r}$, where $X$ is some space and the norm is understood from the context.

Let $\Var{X}$ be a smooth manifold. If $F: \Var{X} \to \R^N$ is a differentiable function on $\Var{X}$, then we denote its derivative at $x\in \Var{X}$ by $\deriv{F}{x}$.

Throughout this paper, $\Var{M}_i \subset \R^N$ denotes an $n_i$-dimensional smooth embedded submanifold of $\R^N$ for $i=1,\ldots,r$, and $\Var{J} = \JOIN(\Var{M}_1, \ldots, \Var{M}_r) \subset \R^N$ denotes their join set. The product manifold $\Var{M}_1 \times \cdots \times \Var{M}_r$ is denoted by~$\Var{P}$. Its dimension is denoted by $n = \sum_{i=1}^r n_i$. The map $\Phi: \Var{P} \to \Var{J}$ is defined as in \refeqn{eqn:phi}.

\section{The condition number of JDPs}\label{sec_condition_number}
This section derives an efficiently computable expression for the condition number of the JDP at a particular decomposition.
As we are dealing with the inverse of a smooth function on a product manifold, the natural starting point of our analysis is the differential-geometric framework of condition \cite{BCSS,BC2013} that applies to smooth maps between manifolds. However, $\Phi$ is not a map between smooth manifolds since $\Var{J}$ is not necessarily a manifold. Indeed, in the specific case that the $\Var{M}_i$ are smooth \emph{algebraic varieties}, i.e., $\Var{M}_i$ is the solution set of a collection of polynomial equations, then the join set is a \emph{semi-algebraic set} by the Tarski--Seidenberg principle \cite{BCR1998} as it corresponds to the projection onto the last factor of the graph of $\Phi$, which is an algebraic set. Such sets generally have \emph{singular points} where the geometric tangent space \cite[Chapter 3]{Lee2013} is not defined. For overcoming this obstacle, the key observation is that we can localize the analysis of \cite[Section 14.3]{BC2013}, \emph{hereby generalizing the differential geometric framework to local inverses of smooth functions on manifolds.} This is performed in the next subsection. In \refsec{sec_relative_condition}, we propose relative condition numbers for the JDP.

\subsection{The local differential-geometric approach} \label{sec_diffgeometric_cn}
The proof of \refthm{thm_condition_number} is presented in this subsection.
We seek to obtain the condition number of the inverse of a smooth function on a manifold. The first helpful result is stated next.
\begin{lemma}\label{lem_kind_of_taylor_series}
Let $\Var{M} \subset \R^N$ be an embedded submanifold. Let $F : \Var{M} \to \R^p$ be a smooth function on $\Var{M}$. Let $\vect{x} \in \Var{M}$. There exist constants $r_F > 0$ and $\gamma_F \ge 0$ such that for all $\vect{y} \in B_{r_F}(\vect{x}) \cap \Var{M}$ we have $\Delta = (\vect{y}-\vect{x}) \in \R^N$ and
\begin{align*}
F(\vect{y}) = F(\vect{x}) + (\deriv{F}{\vect{x}}) P_{\Tang{\vect{x}}{\Var{M}}} \Delta + \vect{v}_{\vect{x},\vect{y}}, \text{ where } \|\vect{v}_{\vect{x},\vect{y}}\| \le \gamma_F \|\Delta\|^2,
\end{align*}
and $P_A$ denotes the orthogonal projection onto the linear subspace $A \subset \R^N$.
\end{lemma}
\begin{proof}
Let $P_\Var{M} : \R^N \to \Var{M}, \vect{z} \mapsto \arg\min_{\vect{y}\in\Var{M}} \|\vect{z}-\vect{y} \|$ be the projection onto the manifold $\Var{M}$, which is a smooth function in the neighborhood of $\vect{x}\in\Var{M}$ \cite[Lemma~4]{AM2012}. By definition we have $\vect{x} = P_{\Var{M}} \vect{x}$ for all $\vect{x} \in \Var{M}$.
By \cite[Proposition 5]{AM2012}, the map $\widehat{R}_{\vect{x}} : \R^N \to \Var{M},\; \eta \mapsto P_{\Var{M}} (\vect{x} + \eta)$ restricted to $\Tang{\vect{x}}{\Var{M}}$ is the \emph{projective retraction}, which is a smooth and well-defined retraction for all $\xi \in \Tang{\vect{x}}{\Var{M}}$ in a neighborhood of $0_\vect{x}$. Let $\tau = \|\Delta\|$ and fix $\eta = \Delta \tau^{-1}$. Consider then $G : \R \to \R^{N},\; t \mapsto F(\widehat{R}_{\vect{x}}(t \eta)).$ By definition it admits a Taylor series approximation
$$G(\tau) = G(0) + \deriv{G}{0} \cdot \tau + \mathcal{O}( \tau^2 ),$$
which is well defined in a neighborhood of~$0$.
We have $G(0) = F(\widehat{R}_{\vect{x}}(0)) = F(\vect{x})$ and $G(\tau) = F(\widehat{R}_{\vect{x}}(\Delta)) = F(P_{\Var{M}} \vect{y}) = F(\vect{y}).$ Moreover, it follows from the chain rule that
$
\deriv{G}{0}
= \deriv{F}{\widehat{R}_{\vect{x}}(0)} \circ \deriv{\widehat{R}_{\vect{x}}}{0} \circ \eta
= \deriv{F}{\vect{x}} \circ ( \deriv{P_{\Var{M}}}{\vect{x}} ) \circ \eta$. Here $\eta$ also denotes the map $t \mapsto t\eta$. As $\deriv{P_{\Var{M}}}{\vect{x}} = P_{\Tang{\vect{x}}{\Var{M}}}$ by \cite[Lemma~4]{AM2012}, we have $\deriv{G}{0}= \deriv{F}{\vect{x}} \circ P_{\Tang{\vect{x}}{\Var{M}}} \circ \eta.$
\end{proof}

The next ingredient we need is a standard fact in differential geometry.
\begin{lemma}\label{lem_local_diffeomorphism}
Let $F : \Var{M} \to \Var{N}$ be a smooth map between manifolds $\Var{M}$ and $\Var{N}$. If the derivative $\deriv{F}{x} : \Tang{x}{\Var{M}} \to \Tang{F(x)}{\Var{N}}$ is injective at $x \in \Var{M}$, then $F$ is a local diffeomorphism onto its image.
\end{lemma}
\begin{proof}
Combine Proposition 4.1, Theorem 4.25 and Proposition 5.2 of \cite{Lee2013}.
\end{proof}

This lemma essentially proves the case in \refthm{thm_condition_number} where $\Phi$ has no unique local inverse, i.e., where the JDP is ill-posed.
\begin{corollary}\label{cor_inf_cond}
If $\Phi$ is not locally invertible at $\tuple{p} \in \Var{P}$, then $\kappa(\tuple{p}) = \infty = \frac{1}{\varsigma_n(\deriv{\Phi}{\tuple{p}})}$.
\end{corollary}
\begin{proof}
If $n = \dim \Var{P} > N$, then $\Phi$ is not a local homeomorphism and hence $\varsigma_n(\deriv{\Phi}{\tuple{p}})=0$. Otherwise, the contrapositive of \reflem{lem_local_diffeomorphism} entails that $\deriv{\Phi}{\tuple{p}}$ cannot be injective. Since $n \le N$ in this case, $\deriv{\Phi}{\tuple{p}}$ is not of full rank $n$.
\end{proof}
We are now ready to prove the general case. The proof strategy is {quite} standard.
\begin{proof}[Proof of \refthm{thm_condition_number}]
Let $\tuple{p} \in \Var{P}$. If $\Phi$ is not locally invertible at $\tuple{p}$, then by \refcor{cor_inf_cond}, $\kappa(\tuple{p})=\infty$. Note that this includes the case $n > N$, hence we can assume $n \le N$ in the remainder. It remains to prove \refthm{thm_condition_number} in the case when there exists a local inverse function $\Phi^{-1}_{\tuple{p}}$ of $\Phi$ at $\vect{p} = \Phi(\tuple{p})$.
In this case, there exists some open neighborhood $\Var{N}$ of $\tuple{p} \in \Var{P}$ such that  $\Phi^{-1}_{\tuple{p}} \circ \Phi|_{\Var{N}} = \operatorname{id}_{\Var{N}}$.
Assume first that $\deriv{\Phi}{\tuple{p}}$ is injective.
It follows from \reflem{lem_local_diffeomorphism} that $\Phi$ is a local diffeomorphism onto its image, i.e., there exists some open neighborhood $\Var{N}' \subset \Var{N}$ of $\tuple{p}$ such that $\Var{J}' := \Phi(\Var{N}') \subset \Var{J}$ is an embedded submanifold of $\Var{J}$ in the subspace topology. We may assume without loss of generality that the domain of $\Phi_{\tuple{p}}^{-1}$ is this manifold $\Var{J}'$. Since $\Phi|_{\Var{N}'}$ is diffeomorphic to $\Var{J}'$, it follows that $\dim \Var{N}' = \dim \Var{J}'$ and that $\deriv{\Phi}{\tuple{p}}$ is invertible.
By applying \reflem{lem_kind_of_taylor_series} to $\Phi_{\tuple{p}}^{-1}$, we find
\begin{align}\label{eqn_classic_0}
\kappa(\tuple{p})
= \kappa[\Phi_{\tuple{p}}^{-1}](\vect{p})
= \lim_{\epsilon\to0} \; \max_{\vect{y} \in (B_{\epsilon}(\vect{p}) \cap \Var{J}')} \frac{\|\deriv{\Phi^{-1}_{\tuple{p}}}{\vect{p}} P_{\Tang{\vect{p}}{\Var{J}'}} (\vect{y} - \vect{p}) \|}{\|\vect{y} - \vect{p}\|} + \mathcal{O}( \| \vect{y} - \vect{p} \| ).
\end{align}
For sufficiently small $\epsilon$, $B_{\epsilon}(\vect{p}) \cap \Var{J}'$ is a path-connected submanifold. Then, for every $\vect{y} \in B_{\epsilon}(\vect{p}) \cap \Var{J}'$ there exist at least one smooth curve $\gamma_{\vect{p}\to\vect{y}}(t)$ in $\Var J'$ connecting $\vect{p}$ and $\vect{y}$. Applying \reflem{lem_kind_of_taylor_series} to this curve, and plugging the result into \refeqn{eqn_classic_0}, we find
\begin{align}\label{eqn_classic_1}
\kappa(\tuple{p}) = \lim_{\epsilon\to0} \; \max_{\vect{y} \in (B_{\epsilon}(\vect{p}) \cap \Var{J}')} \left\|\deriv{\Phi^{-1}_{\tuple{p}}}{\vect{p}}  \frac{\deriv{\gamma_{\vect{p}\to\vect{y}}}{0}}{\|\deriv{\gamma_{\vect{p}\to\vect{y}}}{0}\|} \right\| + \mathcal{O}( \| \vect{y} - \vect{p} \| ).
\end{align}
Since the geometric tangent space is precisely
\[
\Tang{\vect{p}}{\Var{J}'} := \cset{ \deriv{\gamma}{0} \in \R^N }{ \gamma(t) \subset \Var{J}' \text{ is a smooth curve with } \gamma(0)=\vect{p}},
\]
see, e.g., \cite[Proposition 3.23]{Lee2013}, equation \refeqn{eqn_classic_1} is equivalent to
\[
\kappa(\tuple{p})
= \max\limits_{\vect{v} \in \mathbb{S}(\Tang{\vect{p}}{\Var{J}'})} \| \deriv{\Phi^{-1}_{\tuple{p}}}{\vect{p}} \vect{v} \|
= \| \deriv{\Phi^{-1}_{\tuple{p}}}{\vect{p}} \|
= \| ( \deriv{\Phi}{\vect{p}} )^{-1} \|
= \frac{1}{\varsigma_n(\deriv{\Phi}{\tuple{p}})},
\]
where the penultimate step is by invoking the inverse function theorem for manifolds \cite[Theorem 4.5]{Lee2013}. It remains to prove the case where $\deriv{\Phi}{\tuple{p}}$ is not injective. Then, some unit-length vector $\vect{v}\in\Tang{\tuple{p}}{\Var{P}}$ is mapped to $0$. Let $\gamma(t) \subset \Var{N}$ be a smooth curve with $\gamma(0)=\tuple{p}$ and $\deriv{\gamma}{0} = \vect{v}$. From \reflem{lem_kind_of_taylor_series} it follows that
\begin{align} \label{eqn_expansion_we_need}
\Phi(\gamma(t)) - \Phi(\gamma(0)) = (\deriv{\Phi}{\tuple{p}} \circ P_{\Tang{\tuple{p}}{\Var{P}}}) (\gamma(t) - \gamma(0)) + \mathcal{O}(\|\gamma(t)-\gamma(0)\|^2).
\end{align}
Therefore, for sufficiently small $t$,
\begin{align*}
\kappa(\tuple{p})
\ge \lim_{t \to 0} \frac{ \| \Phi_{\tuple{p}}^{-1}(\Phi(\gamma(t))) - \Phi_{\tuple{p}}^{-1}(\Phi(\gamma(0))) \|}{\| \Phi(\gamma(t)) - \Phi(\gamma(0)) \|}
= \lim_{t \to 0} \left( \frac{\| \Phi(\gamma(t)) - \Phi(\gamma(0)) \|}{ \| \gamma(t) - \gamma(0) \|} \right)^{-1},
\end{align*}
where the last step is because $\Phi_{\tuple{p}}^{-1}$ is a local inverse of $\Phi$. Plugging \refeqn{eqn_expansion_we_need} into the last expression yields $\kappa(\tuple{p}) \ge \| (\deriv{\Phi}{\tuple{p}} \circ P_{\Tang{\tuple{p}}{\Var{P}}}) \vect{v} \|^{-1} = \|\deriv{\Phi}{\tuple{p}} \vect{v} \|^{-1} = \infty = \frac{1}{\varsigma_n(\deriv{\Phi}{\tuple{p}})}$, where the last equality is precisely because the derivative is not injective.
This proves the first equality in \refeqn{eqn_main_thm_condition} for all $\tuple{p} \in \Var{P}$.

Finally, we show that $\varsigma_n(\deriv{\Phi}{\tuple{p}}) = \varsigma_n(U)$ for all $\tuple{p} \in \Var{P}$. The derivative of the addition map $\Phi : \Var{P} \to \Var{J} \subset \R^N$ at $\tuple{p}=(p_1,\ldots,p_r)$ is
\begin{align*}
 \deriv{\Phi}{\tuple{p}} : \Tang{p_1}{\Var{M}_1} \times \cdots \times \Tang{p_r}{\Var{M}_r} \to \Tang{\Phi(\tuple{p})}{\R^N},\;
 (\dot{\vect{p}}_1, \ldots, \dot{\vect{p}}_r) \mapsto \dot{\vect{p}}_1 + \cdots + \dot{\vect{p}}_r.
\end{align*}
Letting $U_i$ be as in the theorem, we can uniquely express $\dot{\vect{p}}_i = U_i \vect{x}_i$. Hence,
\[
\deriv{\Phi}{\tuple{p}}( \dot{\vect{p}}_1, \ldots, \dot{\vect{p}}_r) = U_1 \vect{x}_1 + \cdots + U_r \vect{x}_r = \begin{bmatrix} U_1 & \cdots & U_r \end{bmatrix} [\vect{x}_i]_{i=1}^r = U \vect{x}.
\]
The claim follows from the Courant--Fisher characterization of the least singular value, namely
\(
\varsigma_n( \deriv{\Phi}{\tuple{p}} )
= \min_{\tuple{x}\in\bb{S}(\Tang{\tuple{p}}{\Var{P}})} \Norm{(\deriv{\Phi}{\tuple{p}})(\tuple{x})}
= \min_{\vect{x}\in\mathbb{S}(\R^n)} \|U\vect{x}\|
= \varsigma_n(U).
\)
Note that the second equality is valid only for $U_i$'s with orthonormal columns.
\end{proof}

\subsection{Relative condition numbers}\label{sec_relative_condition}
\emph{Relative condition numbers} can be obtained as follows. Let $\tuple{p} = (\vect{p}_1,\ldots,\vect{p}_r)$ and $\tuple{p}' = (\vect{p}_1', \ldots, \vect{p}_r')$ with $\vect{p}_i, \vect{p}_i' \in \Var{M}_i$. Then, it follows from \refeqn{eqn_rule_of_thumb} that
\[
 \| \vect{p}_j - \vect{p}_j' \| \le \left\| \begin{bmatrix} \vect{p}_1 - \vect{p}_1' & \cdots & \vect{p}_r - \vect{p}_r'\end{bmatrix} \right\|_F \lesssim \kappa(\tuple{p}) \cdot \| \Phi(\tuple{p}) - \Phi(\tuple{p}') \|,
\]
provided that $\tuple{p}$ and $\tuple{p}'$ are close in the product metric on $\Var{P}$.
Hence, a relative condition number $\kappa_{j}^\mathrm{rel}$ for the $j$th component of the JDP at $\tuple{p}$ could be defined as
\[
 \kappa_{j}^\mathrm{rel}(\tuple{p}) := \kappa(\tuple{p}) \frac{\|\Phi(\tuple{p})\|}{\|\vect{p}_j\|},
 \quad\text{so that }\quad
 \frac{\|\vect{p}_j - \vect{p}_j'\|}{\|\vect{p}_j\|} \lesssim \kappa_{j}^\mathrm{rel}(\tuple{p}) \frac{\|\Phi(\tuple{p}) - \Phi(\tuple{p}')\|}{\|\Phi(\tuple{p})\|}.
\]

\begin{remark}
One could complimentarily define a relative condition number for the entire decomposition $\tuple{p}$: $\kappa^\mathrm{rel}(\tuple{p}) := \kappa(\tuple{p}) \frac{\|\Phi(\tuple{p})\|}{\|\tuple{p}\|}.$ However, we believe that the former definition will usually be the more informative one, especially when the norms of the rank-$1$ terms appearing in the decomposition $\tuple{p}$ are of different orders of magnitude. Consider the following example. Assume that $\tuple{p}=(\vect{p}_1,\vect{p}_2)$ with $\vect{p}_1,\vect{p}_2\in\R^N$, $\kappa(\tuple{p}) = 1$, $\|\vect{p}_1\| = 1$ and $\|\vect{p}_2\| = 10^{-10}$.\footnote{Such cases can occur for weak $3$-orthogonal tensor rank decompositions; see \refsec{sec_tensor_rank_decomposition}.} Then, $\kappa^\text{rel} \approx 1 \frac{1}{1} = 1$. One might conclude that the whole decomposition is well-conditioned. If we have an alternative decomposition $\tuple{p}'$ with $\|\tuple{p} - \tuple{p}'\| = 10^{-10}$, then the relative forward error is at most about $10^{-10}$. Interpreting this result, it follows that $\tuple{p}'$ could well be $(\vect{p}_1,0)$, which has a positive-dimensional family of decompositions, namely $\tuple{p}_\alpha = (\alpha \vect{p}_1, (1-\alpha) \vect{p}_1)$. All of these decompositions have $\kappa( \tuple{p}_\alpha ) = \infty$. This does not comply well with the intuition that a small relative forward error should imply a numerically stable decomposition. By contrast, the separate relative condition numbers accurately detect that $\vect{p}_2$ is not at all stable to small perturbations, as $\kappa_1^\text{rel} \approx 1 \frac{1}{1} = 1$ and $\kappa_2^\text{rel} \approx 1 \frac{1}{10^{-10}} = 10^{10}$.
\end{remark}

\section{Characterization as a distance to ill-posedness}\label{sec_def_condition_number}
Demmel \cite{Demmel1987} showed that many condition numbers arising in linear algebra can be interpreted as an inverse distance to a locus of ill-posed inputs to the problem.
In this section, we show that the condition number of the JDP is an inverse distance to a locus of ill-posed problems.

Recall that $\Gr(N,m)$ is the Grassmannian manifold of $m$-dimensional subspaces of $\R^N$.
The \emph{projection distance} \cite[Section 2.6]{matrix_computations} between $W,W' \in \Gr(N,m)$ is
\[
\dist(W,W') := \| \Pi_W-\Pi_{W'} \|,
\]
where $\Pi_W$ is the orthogonal projection onto $W$.
For brevity, we introduce the next shorthand for a product Grassmannian:
\[
 \operatorname{Gr}(N,\vect{n}) := \operatorname{Gr}(N,n_1) \times \operatorname{Gr}(N,n_2) \times \cdots \times \operatorname{Gr}(N,n_r),
\]
where $n_i = \dim \Var{M}_i$.
The projection distance between $W=(W_1,\ldots,W_r) \in \Gr(N,\vect{n})$ and $W' = (W_1',\ldots,W_r') \in \Gr(N,\vect{n})$ is then defined as
\begin{equation}\label{proj-distance}
\dist(W,W'):= \sqrt{\sum_{i=1}^r \dist(W_i,W_i')^2}.
\end{equation}
The distance between a point $W$ and a set $\Var{W}$ is defined to be the infimum of the distances from $W$ to any of the elements of the set $\Var{W}$.

Let the \emph{locus of intersecting subspaces} in $\Gr(N,\vect{n})$ be
\begin{equation}\label{GrSigma}
\GrSigma:= \cset{(W_1,\ldots,W_r)\in \Gr(N,\vect{n})}{\dim (W_1+\cdots+ W_r) < n }.
\end{equation}
{We can then define} the locus of \emph{ill-posed solutions} of the JDP as
\begin{align*}
\Sigma := \{ \tuple{p} \in \Var{P} \;|\; \kappa(\tuple{p})=\infty \}
&= \cset{\tuple{p} \in \Var{P}}{\deriv{\Phi}{\tuple{p}} \text{ is not injective}} \\
&= \cset{\tuple{p} \in \Var{P}}{\rank(\deriv{\Phi}{\tuple{p}}) < n} \\
&= \cset{\tuple{p}=(p_1,\ldots,p_r)\in \Var{P}}{(\Tang{p_1}{\Var{M}_1},\ldots,\Tang{p_r}{\Var{M}_r})\in\Sigma_{\Gr}},
\end{align*}
where the second and third equalities follow from the proof of \refthm{thm_condition_number} and the last by \refthm{thm_inverse_distance}.

\begin{remark}
If $\Sigma = \Var{P}$ with $\dim \Var{P} \le N$, then we call $\Var{J} = \JOIN(\Var{M}_1,\ldots,\Var{M}_r)$ \emph{defective} in analogy with defective join and secant varieties of algebraic varieties \cite[Example 11.22]{Harris1992}.
For join sets originating from Segre and Veronese varieties it is known that defective join sets occur only exceptionally; see \cite{AH1995,AOP2009} for details. By definition, the condition number for elements of defective join sets is $\infty$.
\end{remark}

We need two auxiliary results for proving \refthm{thm_inverse_distance}.
Let $\Var{S}^{N\times r} :=\bb{S}(\bb{R}^N)^r \subset \R^{N \times r}$ denote the subset of matrices whose columns are of unit norm, and $\Var{S}^{N \times r}_{<r}$ will denote the subset of $\Var{S}^{N \times r}$ of matrices whose rank is strictly less than $r$. The first lemma is an alternative characterization of the smallest singular value of a matrix with unit-norm columns.
\begin{lemma}\label{lem_matrixnorms3}
Let $r \le N$, and let $Y \in \Var{S}^{N\times r}$. Then,
\begin{align}\label{eqn_singvalchar}
\varsigma_{r} (Y)
= \min\limits_{\stackrel{X\in \bb{R}^{N\times r},}{\rank(X) < r}} \;\sqrt{\sum_{i=1}^r \Norm{\vect{x}_i-\vect{y}_i}^2}
= \min\limits_{\stackrel{X\in \bb{R}^{N\times r},}{\rank(X) < r}} \;\sqrt{\sum_{i=1}^r \bigl( \sin \sphericalangle(\vect{x}_i, \vect{y}_i) \bigr)^2}.
\end{align}
\end{lemma}
\begin{proof}
The first equality is by the Eckart-Young characterization of the smallest singular value.
Let $X$ be a minimizer of the middle expression in \refeqn{eqn_singvalchar}. For every $1 \le i \le r$, we distinguish between two cases. If $\vect{x}_i\neq 0$, let $\vect{x}_i' = \alpha \vect{x}_i$ denote the orthogonal projection of $\vect{y}_i$ onto $\vect{x}_i$. By definition, $\Norm{\vect{x}_i-\vect{y}_i}\geq \Norm{\vect{x}_i'-\vect{y}_i} = \sin \sphericalangle (\vect{x}_i, \vect{y}_i)$. Otherwise, if $\vect{x}_i=0$, we choose $\vect{x}_i'$ as any nonzero vector in $\mathrm{span}\set{\vect{x}_1,\ldots,\vect{x}_r}\cap (\vect{y}_i)^\perp$. Then, $\Norm{\vect{x}_i - \vect{y}_i} = 1 = \sin\sphericalangle(\vect{x}_i', \vect{y}_i)$ by construction. In both cases, the column span of $X$ does not change if we replace the $i$th column by $\vect{x}_i'$. Hence, $X' = \left[\begin{smallmatrix} \vect{x}_1' & \cdots & \vect{x}_r' \end{smallmatrix}\right]$ is of rank $< r$. Hence, $\varsigma_r(Y)$ is bounded from below by the right-hand side of \refeqn{eqn_singvalchar}.

Conversely, let $X' = \left[\begin{smallmatrix} \vect{x}_1' & \cdots & \vect{x}_r' \end{smallmatrix}\right]$ be a minimizer of the right-hand side of \refeqn{eqn_singvalchar}. Let $\vect{x}_i$ be the orthogonal projection of $\vect{y}_i$ onto $\vect{x}_i'$. Then $X = \left[\begin{smallmatrix} \vect{x}_1 & \cdots & \vect{x}_r \end{smallmatrix}\right]$ is of rank $<r$ with $\Norm{\vect{x}_i-\vect{y}_i} = \sin \sphericalangle (\vect{x}_i', \vect{y}_i)$, proving the converse inequality.
\end{proof}

The second auxiliary lemma sufficiently characterizes $\Sigma_\Gr$ for our purpose.
\begin{lemma}\label{characterization_of_sigma_gr}
Let $n=n_1+\cdots+n_r$. For $X \in \Var{S}^{N\times r}$, we define
\[
\calZ{E}_X :=\cset{(W_1,\ldots,W_r)\in\Gr(N,\vect{n})}{\forall 1\leq i \leq r: \vect{x}_i \in W_i}.
\]
Then, the following statements hold:
\begin{enumerate}
\item $\Sigma_\Gr$ is the union of $\calZ{E}_X$ over all $X\in\Var{S}^{N \times r}_{< r}$ of rank strictly less than $r$; and
\item $\Sigma_\Gr$ is an algebraic subvariety of $\Gr(N,\tuple{n})$.
\end{enumerate}
\end{lemma}

\begin{proof}
Part (1) is a direct consequence of the observation
\begin{align*}
W\in\GrSigma \;\text{ iff } \; \exists X \in \Var{S}^{N\times r}_{<r}: W\in\calZ{E}_X, \quad\text{and}\quad
\forall X\in\Var{S}^{N \times r}_{<r} : \calZ{E}_X \subset \Sigma_\Gr.
\end{align*}
To prove (2) we mildly generalize \cite[Example 8.30]{Harris1992}: Since $\Gr(N,m) \simeq \Pj( \wedge^m \R^N)$ via the Pl\"ucker embedding \cite[Chap. 3.1.C]{gkz}, we see that
 \[
  \Sigma_\Gr = \Bigl\{([\vect{w}_{1,1} \wedge \cdots \wedge \vect{w}_{1,n_1}],\ldots,[\vect{w}_{r,1} \wedge \cdots \wedge \vect{w}_{r,n_r}]) \;|\; 0 = \bigwedge_{i=1}^r \bigwedge_{j=1}^{n_i} \vect{w}_{i,j} \Bigr\} \subset \Gr(N,\tuple{n})
 \]
is a projective subvariety cut out set-theoretically by the above equation.
\end{proof}

An important consequence of \reflem{characterization_of_sigma_gr}(2) is that $\Sigma_\Gr$ is closed in the Euclidean topology. Hence, for all $W\in\Gr(N,\vect{n})$ there exists a closest $W'\in\Sigma_\Gr$ with $\dist(W,W')=\dist(W,\Sigma_\Gr)$; that is, the infimum is attained.

Now we can prove the next result which essentially entails \refthm{thm_inverse_distance}.
\begin{proposition}\label{prop1}
Let $W=(W_1,\ldots,W_r)\in\Gr(N,\vect{n})$. Let $U_i$ be a matrix whose columns form an orthonormal basis for $W_i$, and set $U = \begin{bmatrix} U_1 & \cdots & U_r \end{bmatrix} \in \R^{N\times n}$. Then, we have $
\dist(W,\GrSigma) = \varsigma_n(U).
$
\end{proposition}
\begin{proof}
We show first that \(\dist(W,\GrSigma) \le \varsigma_n(U)\). It is easy to verify that
\begin{align}
 \varsigma_n(U)
:= \min\limits_{\hat{\vect{z}}\in \bb{S}(\R^n)} \Norm{U \hat{\vect{z}}}
\label{eqn_signchar} &=  \min\limits_{{\vect{w}_i \in\bb{S}(W_i),\, 1 \le i \le r}} \; \min\limits_{\vect{z}\in \bb{S}(\bb{R}^r)}\; \Norm{z_1 \vect{w}_1 +\cdots +z_r \vect{w}_r} \\
\nonumber &=  \min\limits_{{\vect{w}_i \in\bb{S}(W_i),\, 1 \le i \le r}} \; \varsigma_r \bigl(\begin{bmatrix}\vect{w}_1 & \cdots & \vect{w}_r \end{bmatrix}\bigr).
\end{align}
Let $Y^\star = \begin{bmatrix}\vect{y}_1^\star & \cdots & \vect{y}_r^\star\end{bmatrix}$ with $\vect{y}_i^\star\in\bb{S}(W_i)$ be a minimizer of the last statement in the above expression, so that $\varsigma_n(U)= \varsigma_r(Y^\star)$. From the Eckart--Young theorem, we find
\[
 \varsigma_n(U) = \varsigma_r(Y^\star) = \min_{{X \in \R^{N \times r},\, \rank(X)<r}} \| X - Y^\star \|;
\]
let $X^\star$ be an optimizer of this expression. Then,
\begin{align*}
 \varsigma_n(U) = \varsigma_r(Y^\star)
\ge \min\limits_{\substack{\vect{w}_i \in\bb{S}(W_i),\\ 1 \le i \le r}} \sqrt{ \sum_{i=1}^r \| \vect{x}_i^\star - \vect{w}_i \|^2 }
&\ge \min\limits_{\substack{\vect{w}_i \in\bb{S}(W_i),\\ 1 \le i \le r}} \sqrt{ \sum_{i=1}^r \| \Pi_{\vect{x}_i^\star}(\vect{w}_i) - \vect{w}_i \|^2 }.
\end{align*}
From \cite[Lemma 3.2]{amelunxen1}, we have
\begin{align} \label{eqn_lem_amelunxen}
\dist(W_i, \calZ{E}_{\vect{x}_i^\star}) = \min\limits_{\vect{w}_i\in \bb{S}(W_i)} \; \sin\sphericalangle(\vect{x}_i^\star,\vect{w}_i) = \min\limits_{\vect{w}_i\in \bb{S}(W_i)} \| \Pi_{\vect{x}_i^\star}(\vect{w}_i) - \vect{w}_i \|,
\end{align}
where $\calZ{E}_{\vect{x}_i^\star} := \cset{W'\in\Gr(N,n_i)}{\vect{x}_i^\star \in W'}$. Hence, by the definition of $\dist(W, \calZ{E}_X)$ and $\calZ{E}_X$, one finds $\varsigma_r(Y^\star) \ge \dist(W, \calZ{E}_X)$. Since $\calZ{E}_X \subset \GrSigma$ by \reflem{characterization_of_sigma_gr}(1), we have proven the first bound: $\dist(W, \GrSigma) \le \dist(W, \calZ{E}_X) \le \varsigma_n(U)$.

For proving the converse inequality, let $W'=(W_1',\ldots,W_r')\in\GrSigma$ be such that $\dist(W,\GrSigma)= \dist(W,W')$. By \reflem{characterization_of_sigma_gr}(1) there exists a matrix $X \in \Var{S}^{N\times r}_{<r}$, such that $W'\in\calZ{E}_X$. By definition~\refeqn{proj-distance} we have
\[
\dist(W,\GrSigma)= \dist(W,W')
= \Big(\sum_{i=1}^r \Norm{\Pi_{W_i}-\Pi_{W_i'}}^2\Big)^\frac{1}{2}
\geq \Big(\sum_{i=1}^r \Norm{(\Pi_{W_i}-\Pi_{W_i'}) \vect{x}_i}^2\Big)^\frac{1}{2}.
\]
For $1\leq i\leq r$, let $\vect{w}_i =\Pi_{W_i} \vect{x}_i$ and $\vect{y}_i:= \frac{\vect{w}_i}{\Norm{\vect{w}_i}} \in W_i$ if $\vect{w}_i\neq 0$, or, if $\vect{w}_i=0$, let $\vect{y}_i$ denote some arbitrary but fixed point in $\bb{S}(W_i)$. Then for each $i$, we have $\Norm{(\Pi_{W_i}-\Pi_{W_i'})\vect{x}_i} =\Norm{\vect{w}_i-\vect{x}_i}= \sin \sphericalangle(\vect{x}_i, \vect{w}_i) = \sin \sphericalangle(\vect{x}_i, \vect{y}_i)$, and, hence,
\begin{align*}
	\dist(W,\GrSigma)
	\geq \Big(\sum_{i=1}^r \left(\sin \sphericalangle(\vect{x}_i, \vect{y}_i)\right)^2\Big)^\frac{1}{2}
 	\geq  \varsigma_r\bigl(\begin{bmatrix}\vect{y}_1 & \cdots & \vect{y}_r\end{bmatrix} \bigr)
	\geq \varsigma_n(U),
\end{align*}
where the second step is due to \reflem{lem_matrixnorms3} and the last step because of \refeqn{eqn_signchar}.
\end{proof}

We are now ready to wrap up the proof of \refthm{thm_inverse_distance}.

\begin{proof}[Proof of \refthm{thm_inverse_distance}]
Let $\tuple{p}=(p_1,\ldots,p_r)\in \Var{M}_1\times \cdots\times\Var{M}_r$. Applying \refprop{prop1} to $W := (\Tang{p_1}{\Var{M}_1}, \ldots, \Tang{p_r}{\Var{M}_r}) \in \operatorname{Gr}(N,\vect{n})$ concludes the proof.
\end{proof}

\section{The condition number of open boundary points}\label{sec_boundary_points}
As mentioned in the introduction, in applications one seeks to solve the JAP $
\min_{\tuple{p} \in \Var{P}} \,\tfrac{1}{2} \| \Phi(\tuple{p}) -~\widetilde{\vect{p}} \|^2,
$
where~$\widetilde{\vect{p}} \in \R^N$ and $\Var{P} := \Var{M}_1 \times \cdots \times \Var{M}_r$. Since a closed solution of this optimization problem is unknown, iterative optimization methods are usually employed.

A serious complication occurs in these methods when the JAP has no minimizer.
Let $\Var{J} \subset \R^N$ be a join set, and let $\overline{\Var{J}}$ denote its closure in the Euclidean topology. {We call $\mathcal{B} := \overline{\Var{J}}\setminus\Var{J}$ the set of \emph{open boundary points} of $\Var{J}$. Denote by $\Xi \subset \R^N$ the set of all inputs to the JAP such that the infimum cannot be attained; they are ill-posed inputs. It is easy to show that $\Var{B} \subset \Xi$, and  $\Xi = \emptyset$ if and only if $\Var{B} = \emptyset$.
If ${\vect{p}} \in \Xi$ is an ill-posed input, then there is no solution to the JAP, but aforementioned optimization methods will nevertheless produce an approximate minimizer $\tuple{p}_\star \in \Var{P}$. While the backward error $\|\Phi(\tuple{p}_\star) - \vect{p}\|$ may be small, the components $\tuple{p}_\star$ of the computed approximate solution usually admit no interpretation, since $\vect{p} \not\in \Var{J}$ does not satisfy the model!}

\begin{remark} \label{rem_open_boundary_exists}
The existence of open boundary points is not an exceptional phenomenon. Consider the case of join sets of projective varieties. Chapter 10 of \cite{Landsberg2012} contains examples of join sets $\Var{J} = \JOIN(\Var{M},\ldots,\Var{M})$ where the closure of $\Var{J}$ in the Euclidean topology, write $\overline{\Var{J}}$, is strictly larger than $\Var{J}$; namely, this situation arises for tensor rank decompositions, Waring decompositions, and partially symmetric decompositions. In the context of secant varieties, points in $\overline{\Var{J}}\setminus\Var{J}$ are said to admit an \emph{$\Var{M}$-border rank} \cite[Section 5.2]{Landsberg2012} that is strictly less than their rank.
\end{remark}

This section is devoted to proving \refthm{thm_boundary_points_cn} which states that all decompositions $\tuple{p}_\star \in \Var{P}$ whose corresponding point in the join set $\vect{p}_\star = \Phi(\tuple{p}_\star)$ is close to the locus of ill-posed inputs $\Xi$ must necessarily have a large condition number. This can then be interpreted in light of \refeqn{eqn_rule_of_thumb}, which would reveal that $\tuple{p}_\star$ is not a stable decomposition. The condition number of the computed approximation $\tuple{p}_\star$ is thus always a reliable measure of stability, even for ill-posed inputs $\Xi$ of the JAP.

Unfortunately, we could not prove \refthm{thm_boundary_points_cn} without additional assumptions on the manifolds $\Var{M}_1, \ldots, \Var{M}_r$ and their join set $\Var{J} = \JOIN(\Var{M}_1, \ldots,\Var{M}_r)$. The two extra assumptions that we make are that all of the manifolds $\Var{M}_i$ are \emph{cones}, and that
$$
\Var{J} = \JOIN(\Var{M}_1,\ldots,\Var{M}_r) = \JOIN(\overline{\Var{M}}_1,\ldots,\overline{\Var{M}}_r),
$$
where the overline denotes the Euclidean closure. Recall that we call a set $C \subset \R^N$ a cone if for every $p \in C$ also $\alpha p \in C$ for all $\alpha \in \R_0$.

It is natural to wonder whether aforementioned assumptions are necessary. We give two examples\footnote{We would like to thank an anonymous referee for pointing out the first example to us.} of join sets where one of these assumptions is not satisfied, and show that the conclusion of \refthm{thm_boundary_points_cn} does not hold. The first example features an open boundary point for which the condition number converges to a finite value.

\begin{example}\label{ex1}
Consider the following smooth manifolds embedded in $\R^3$:
\[
\Var{M}_1=(\R\backslash\{0\}) \times \R \times \{0\} \quad\text{and}\quad \Var{M}_2=\{0\}\times \{0\} \times \R.
\]
These manifolds violate the second assumption because $(0,1,1)$ is in \(\JOIN(\overline{\Var{M}}_1, \overline{\Var{M}}_2)\) but not in \(\JOIN(\Var{M}_1, \Var{M}_2).\)

Let $p_1(t) = (\sin(t),\cos(t),0)$ and $p_2(t)=(0,0,1)$ for $0< t <\pi$. Then, we have $$\lim_{t\to 0} (p_1(t)+p_2(t))=(0,1,1)\not\in \JOIN(\Var M_1,\Var M_2).$$
The tangent spaces to $\Var M_1$ and $\Var M_2$ are orthogonal to each other. Using this information, we can determine the condition number easily via \refthm{thm_condition_number}: the matrix  $U$ in the statement of the theorem is a matrix with orthonormal columns, so that $\varsigma_2(U) = 1$. Hence,
\(
 \kappa( (p_1(t), p_2(t)) ) = 1^{-1}.
\)
Consequently, $\lim_{t\to 0} \kappa(p_1(t),p_2(t)) = 1$.
\end{example}

The second example shows a join set that has an open boundary point where the condition number neither converges nor diverges towards $\infty$.

\begin{example}\label{ex2}
Consider the embedded smooth submanifolds of $\R^3$,
\[
\Var{M}_1=\cset{(-t,0,0)}{1 \le t < \infty}\quad\text{and}\quad \Var{M}_2=\cset{\big(t,\tfrac{\cos(t)}{t},\tfrac{\sin(t^2)}{t}\big)}{1 \le t < \infty }.
\]
It is easy to see that they are not cones.

Let $p_1(t) = (-t,0,0)$ and $p_2(t)=\big(t,\tfrac{\cos(t)}{t},\tfrac{\sin(t^2)}{t}\big)$ for $1\leq t\leq \infty$. Then,
$$
\lim\limits_{t\to \infty} p_1(t)+p_2(t) = (0,0,0),
$$
which is not a point in $\JOIN(\Var{M}_1,\Var{M}_2)$. The tangent spaces to $\Var M_1$ and $\Var M_2$ are given respectively by the spans of
\begin{align*}
\vect{v}(t) = (1,0,0) \text{ and }
\vect{w}(t) = \left(1,\tfrac{-\sin(t)t-\cos(t)}{t^2},2\cos(t^2) - \tfrac{\sin(t^2)}{t^2}\right).
\end{align*}
By \refthm{thm_condition_number}, $
 \kappa( (p_1(t), p_2(t)) )^{-1} = \varsigma_2 \bigl( \begin{bmatrix} \vect{v}(t) & \vect{w}(t) / \|\vect{w}(t)\| \end{bmatrix} \bigr) = \varsigma_2 ( U(t) ).$ Let $z(t)$ be such that $\|\vect{w}(t)\|^2 = 1 + z(t)$. It is straightforward to show that
\[
 \varsigma_2 ( U(t) ) = \sqrt{1 - \frac{1}{\sqrt{1+z(t)}}} = \sqrt{2} \sin\left( \tfrac{\theta(t)}{2} \right),
\]
where $\theta(t) = \arccos \bigl( ( 1 + z(t) )^{-\frac{1}{2}} \bigr)$ happens to be the angle between $\Tang{p_1(t)}{\Var{M}_1}$ and $\Tang{p_2(t)}{\Var{M}_2}$. For large $t$ we have $\theta(t)\sim \arccos\bigl( (1+(2 \cos(t^2))^2 )^{-\frac{1}{2}} \bigr)$, from which we see that $\kappa( (p_1(t),p_2(t)) )$ does not converge.

Note that in this example there exist both convergent subsequences as well as divergent subsequences.
Indeed, letting $t_k = \sqrt{2 k \pi}$, one finds
\[
\lim_{k\to\infty} \kappa( (p_1(t_k), p_2(t_k)) )
= \frac{1}{\sqrt{2} \sin( \arccos\bigl( \tfrac{1}{\sqrt{5}} \bigr) )}
= \sqrt{1-\sqrt{\tfrac{1}{5}}},
\]
whereas letting $t_k = \sqrt{k \pi / 2}$ yields
\[
\lim_{k\to\infty} \kappa( (p_1(t_k), p_2(t_k)) )
\to \frac{1}{\sqrt{2} \sin( \arccos\bigl(1 \bigr) )}
= \infty.
\]
This situation is troublesome because it is a theoretical possibility that an iterative algorithm for solving the associated JAP with $\vect{p} = \vect{0} \in \overline{\Var{J}}\setminus\Var{J}$ as input yields only iterates for which the condition number is bounded  by some small constant, so that the condition number cannot detect that the final iterate has no meaningful interpretation.
\end{example}

The examples illustrate that additional conditions on the manifolds $\Var{M}_i$ need to be imposed to continuously extend the condition number to open boundary points. However, we do not know which are the minimal restrictions for reaching this goal.

We proceed with the proof of \refthm{thm_boundary_points_cn}.
The key property that we exploit is the fact that the condition number of the JDP at $(p_1,\ldots,p_r)$ is invariant under scaling of the individual factors if the $\Var{M}_i$'s are cones. The reason for this is that the tangent space to a point of a cone is invariant under scaling of that point.
\begin{proposition}\label{scale_invariance}
Let $\Var{M}_i$ be an $n_i$-dimensional smooth submanifold embedded in $\R^N$. Assume that $\Var{M}_i$ is a cone. Let $p_i \in \Var{M}_i$, then
\[
 \kappa((p_1, p_2, \ldots, p_r)) = \kappa( (\beta_1 p_1, \beta_2 p_2, \ldots, \beta_r p_r) ) \text{ for all } \beta_i \in \R_0, i = 1,2,\ldots,r.
\]
\end{proposition}

We can now present the proof of \refthm{thm_boundary_points_cn}.

\begin{proof}[Proof of \refthm{thm_boundary_points_cn}]
For brevity, let $\tuple{p}(t):=(\vect{p}_1(t),\ldots,\vect{p}_r(t))$.
Assume that the curve $\tuple{p}(t)$ {admits a bounded subsequence, then this subsequence} has a limit $\tuple{p}_\star \in\overline{\Var{M}_1}\times\ldots\times\overline{\Var{M}_r}$. As $\Phi$ is a continuous map and $\Var J= \JOIN(\overline{\Var M_1},\ldots, \overline{\Var M_r})$, it follows that $ \vect{p}_\star = \Phi(\tuple{p}_\star) \in \Var{J}$, which {contradicts the assumption in the formulation of the theorem.}

The foregoing shows that $\tuple{p}(t)$ becomes unbounded, i.e., $\lim_{t\to0} \Norm{\tuple{p} (t)} \to \infty$, in the product metric on $\Var{M}_1\times\ldots\times\Var{M}_r$, i.e.,
\(
 \| \tuple{p}(t) \| = \sqrt{ \sum_{i=1}^r \| \vect{p}_i(t) \|^2 }.
\)
Hence, there is at least one $i_\star \in [1,d]$ such that $\lim_{t\to0} \|\vect{p}_{i_\star}(t)\| \to \infty$. The sequence $\vect{p}(t) := \Phi(\tuple{p}(t))$ is assumed to converge in $\R^N$, so it is bounded for small $t$. Since $\vect{p}(t)=\vect{p}_1(t)+\cdots+\vect{p}_r(t)$, there exists another $j_{\star} \in [1,d]$ with $j_\star \ne i_\star$ so that $\lim_{t\to0} \|\vect{p}_{j_\star}(t)\| \to \infty$.

Let $U(t) = [U_1(t), \ldots,U_r(t)]$, where $U_{i}(t)$ is a matrix with orthonormal columns that form a basis for $\Tang{\vect{p}_i(t)}{\Var{M}_i}$. Then \refthm{thm_condition_number} states that
$\kappa(\tuple{p}(t)) = \frac{1}{\varsigma_n( U(t) )}$. For proving that $\lim_{t\to0} \kappa( \tuple{p}(t) ) \to \infty$ we will construct an explicit $\vect{x}(t)\in \Tang{\vect{p}_1(t)}{\Var{M}_1} \times \cdots \times \Tang{\vect{p}_r(t)}{\Var{M}_r}$  and show $\|U(t) \vect{x}(t)\|\to 0$, while $\|\vect{x}(t)\|\geq 1$, implying $\varsigma_n(U(t))\to 0$.

Let $\vect{q}(t),\widetilde{\vect{q}}(t)\subset \bb{S}(\bb{R}^N)$ be the two curves with
\begin{equation}\label{d0}
\alpha(t) \vect{q}(t)= \vect{p}_{i_\star}(t),\; \alpha(t) > 0, \quad \text{ and } \quad  \beta(t) \widetilde{\vect{q}}(t) = \sum_{j \ne i_\star} \vect{p}_j(t),\; \beta(t) > 0.
\end{equation}
Since $\alpha(t) \vect{q}(t)+\beta(t)\widetilde{\vect{q}}(t)$ converges for $t\to 0$, there is a constant $c>0$ with
\begin{equation*} 
c\geq \Norm{\alpha(t) \vect{q}(t)+\beta(t)\widetilde{\vect{q}}(t)}^2 = \alpha(t)^2 + 2\alpha(t) \beta(t) \langle \vect{q}(t), \widetilde{\vect{q}}(t) \rangle + \beta(t)^2.
\end{equation*}
Since $\alpha(t)$ is unbounded and bounded away from $0$ for $t \approx 0$, $\beta(t)$ is unbounded as well. Write $\beta(t) = \gamma(t) \alpha(t)$ with $\gamma(t) > 0$. Then, we have
\begin{equation}\label{d2}
c  \alpha(t)^{-2} \geq 1 + 2 \gamma(t) \langle \vect{q}(t), \widetilde{\vect{q}}(t) \rangle + \gamma(t)^2 .
\end{equation}
Real solutions $\gamma(t)$ exist only if $\langle \vect{q}(t), \widetilde{\vect{q}}(t) \rangle^2 \ge 1 - c \alpha(t)^{-2}$. In fact, by the Cauchy-Schwartz inequality $\langle \vect{q}(t), \widetilde{\vect{q}}(t) \rangle^2 \leq 1$, the only solution consistent with \refeqn{d2} satisfies
\begin{equation*} 
 \lim\limits_{t\to0} \langle \vect{q}(t), \widetilde{\vect{q}}(t) \rangle = -1.
\end{equation*}
Similar to \refeqn{d0}, for all $j\neq i_\star$, write $\vect{p}_j(t)= \nu_j(t) \vect{q}_j(t)$ so that $\vect{q}_j(t)\in\bb{S}(\bb{R}^N)$ and $\nu_j(t) \ge 0$. Since the $\Var{M}_i$'s are all cones we have $\vect{q}_i(t)\in \Var{M}_i$ and, because of homogeneity, $\Tang{\vect{p}_i(t)}{\Var{M}_i} = \Tang{\vect{q}_i(t)}{\Var{M}_i}$ and $\vect{q}_i(t) \in \Tang{\vect{q}_i(t)}{\Var{M}_i}$. For all $1\leq i\leq r$ we write $\vect{q}_{i}(t) = U_{i}(t) \vect{x}_{i}(t).$ Consequently, $\|\vect{x}_i(t)\|=1$. Then
\[
 \widetilde{\vect{q}}(t) = \beta(t)^{-1} \sum_{j\ne i_\star} \nu_j(t) \vect{q}_j(t) = \beta(t)^{-1} \sum_{j\ne i_\star} \nu_j(t) U_{j}(t) \vect{x}_j(t).
\]
Let $(\vect{z}_1,\ldots,\vect{z}_k) := [\vect{z}_1^T \;\ldots\; \vect{z}_k^T ]^T$ denote the vertical stacking of the vectors. Then,
\begin{multline*}
\beta(t) \vect{x}(t) = \\
\Bigl(
\nu_1(t) \vect{x}_1(t),
\ldots,
\nu_{i_\star-1}(t) \vect{x}_{i_\star-1}(t),
\beta(t) \vect{x}_{i_\star}(t),
\nu_{i_\star+1}(t) \vect{x}_{i_\star+1}(t),
\ldots,
\nu_r(t) \vect{x}_r(t) \Bigr).
\end{multline*}
For all $t \approx 0$ we have $\|\vect{x}(t)\|\ge \|\vect{x}_{i_\star}(t)\|=1$. Since $\langle \vect{q}(t), \widetilde{\vect{q}}(t) \rangle$ tends to $-1$ as $t\to0$, it follows by construction that $\lim_{t\to0} U(t)\vect{x}(t) = 0$.
\end{proof}

Let $\Var{X}$ be a smooth projective variety. Then, combining the foregoing theorem with \cite[Theorem 2.1]{SS2017} shows that all points on the algebraic boundary of the real rank-$2$ boundary $\rho(\Var{X})$ inside of $\JOIN(\Var{X},\Var{X})$ are always ill-posed for the JDP.

We conclude this section by showing that the several join sets arising in tensor decomposition problems satisfy the assumptions of \refthm{thm_boundary_points_cn}.

\begin{proposition}
Let $\Var{X} \subset \mathbb{P}\mathbb{R}^N$ be a smooth projective variety, and let $\widehat{\Var{X}} \subset \R^N$ be the affine cone over $\Var{X}$, excluding zero. Then, $\Sec{r}{\widehat{\Var{X}}} = \JOIN(\widehat{\Var{X}},\ldots,\widehat{\Var{X}})$ satisfies the assumptions in \refthm{thm_boundary_points_cn}.
\end{proposition}
\begin{proof}
The fact that $\widehat{\Var{X}}$ is a smooth, analytic manifold is by definition, because the vertex at zero, which could be singular, is removed from our definition of a cone. The closure of $\widehat{\Var{X}}$ in the Euclidean topology is $\overline{\Var{X}} := \widehat{\Var{X}} \cup \{0\}$. To show that $\Sec{r}{\overline{\Var{X}}} \subset \Sec{r}{\widehat{\Var{X}}}$ (the converse inclusion is trivial), let $\vect{p} = \vect{p}_1 + \cdots + \vect{p}_r$ with $\vect{p}_i \in \overline{\Var{X}}$.
Assume without loss of generality that the first $k$, $1 \le k \le r$, $\vect{p}_i$'s are zero. If $k=1$, take $\vect{q}=\vect{p}_2$ and then $\vect{p} = 2 \vect{p}_2 - \vect{p}_2 + \vect{p}_3 + \cdots + \vect{p}_r \in \Sec{r}{\widehat{\Var{X}}}$ as well. If $k > 1$, take $\vect{q} \in \widehat{\Var{X}}$ arbitrary and note that $\sum_{i=1}^k \vect{p}_i = 0 = \sum_{i=1}^k \alpha_i \vect{q}$ with $\sum_{i=1}^k \alpha_i = 0$ and all $\alpha_i \ne 0$.
Hence, $\vect{p} = \vect{p}_1 + \cdots + \vect{p}_r = \alpha_1 \vect{q} + \cdots + \alpha_k \vect{q} + \vect{p}_{k+1} + \cdots + \vect{p}_r \in \Sec{r}{\widehat{\Var{X}}}$ as well.
\end{proof}

In particular, this result covers the next $r$-secant sets: tensors of canonical rank bounded by $r$; symmetric tensors of Waring rank bounded by $r$; and partially symmetric tensors of partially symmetric rank bounded by $r$.
However, the proposition does not apply to block term decompositions, because the set of tensors of multilinear rank at most $(r_1,\ldots,r_d)$ is not a smooth projective variety \cite[Section 3]{LW2007}. We do not know if a result analogous to \refthm{thm_boundary_points_cn} holds for block term decompositions.

\section{Examples} \label{sec_examples}
This section presents explicit expressions for the matrices $U_i$ in \refthm{thm_boundary_points_cn} for two well-studied tensor-related JDPs, so that the condition number can be computed efficiently by computing the least singular value of $\left[ \begin{smallmatrix} U_1 & U_2 & \cdots & U_r \end{smallmatrix} \right].$
The first example is the \emph{tensor rank decomposition problem} that consists of computing the rank-$1$ tensors comprising the CP decomposition\footnote{This decomposition is also called CANDECOMP/PARAFAC, canonical polyadic decomposition, or tensor rank decomposition in the literature.} (CPD) of an $m_1 \times \cdots \times m_d$ tensor. The second example is the problem of computing the symmetric rank-$1$ tensors appearing in a Waring decomposition of a symmetric $m \times \cdots \times m$ tensor.

\subsection{The CP decomposition}
\label{sec_tensor_rank_decomposition}
Recall that the Segre manifold $\Var{S} \subset \R^{N}$, where $N = m_1\cdots m_d$, is the analytic manifold of rank-1 tensors.\footnote{For consistency with the rest of the article, we exploit the natural isomorphism between $\R^{m_1 \times \cdots \times m_d}$ and $\R^N$, given by \emph{vectorizing} the tensor.} Let $\otimes$ denote the tensor product. Then, the $r$-secant set of $\Var{S}$ is
\[
 \Sec{r}{\Var{S}} = \JOIN(\Var{S}, \ldots, \Var{S}) = \Bigl\{ \sum_{i=1}^{r} \sten{a}{i}{1} \otimes \cdots \otimes \sten{a}{i}{d} \;|\; \sten{a}{i}{k} \in \R^{m_k}_0 \Bigr\},
\]
which is also a semi-algebraic set \cite{dSL2008}. This class appears in a wide variety of applications, such as psychometrics \cite{Kroonenberg2008}, chemical sciences \cite{SBG2004}, theoretical computer science \cite{BCS1997}, signal processing \cite{Review2017}, and machine learning \cite{Review2017}, among others. In several of these applications, the uniquely determined rank-$1$ terms appearing in this CPD admit an interpretation in the application domain; see \cite[Section IX]{Review2017}. Hence, the condition number is of natural interest \cite{V2017}.

In the remainder of this section, let $p_i = \mu_i \sten{a}{i}{1} \otimes \cdots \otimes \sten{a}{i}{d} \in \Var{S}$ with $\sten{a}{i}{k} \in \mathbb{S}(\R^{m_k})$ and~$\mu_i > 0$ be rank-$1$ tensors. Then, an orthogonal basis for $\Tang{p_i}{\Var{S}} \subset \R^{N}$ is obtained by considering the derivative of the surjective map
\[
 \Psi : \R \times \mathbb{S}(\R^{m_1}) \times \cdots \times \mathbb{S}(\R^{m_d}) \to \Var{S}, \quad (\mu_i, \sten{a}{i}{1}, \ldots, \sten{a}{i}{d}) \mapsto \mu_i \sten{a}{i}{1} \otimes \cdots \otimes \sten{a}{i}{d},
\]
where $\otimes$ should be interpreted as the Kronecker product.
Hence, the basis vectors are given by the columns of
\begin{align} \label{eqn_explicit_tg_space}
U_i :=
\begin{bmatrix}
\sten{a}{i}{1} \otimes \cdots \otimes \sten{a}{i}{d} &
Q_i^1 \otimes \sten{a}{i}{2} \otimes \cdots \otimes \sten{a}{i}{d} &
\cdots &
\sten{a}{i}{1} \otimes \cdots \otimes \sten{a}{i}{d-1} \otimes Q_i^d
\end{bmatrix},
\end{align}
where $Q_i^k \in \R^{m_k \times (m_k-1)}$ is a matrix whose columns form an orthonormal basis of $\Tang{\sten{a}{i}{k}}{\mathbb{S}(\R^{m_k})}$, i.e., the orthogonal complement of $\sten{a}{i}{k}$.
From \refthm{thm_condition_number} it follows that the condition number of the CPD problem at $\tuple{p} = (p_1,\ldots,p_r)$ is
\(
 \kappa(\tuple{p}) = ( \varsigma_{n} ( \left[\begin{smallmatrix} U_1 & \cdots & U_r \end{smallmatrix}\right] ) )^{-1},
\)
where $n = \dim \Var{S}^{\times r} = r \cdot \dim \Var{S} = r(1 - d + \sum_{k=1}^d m_k)$.

Using substantially the same arguments as in \cite[Section 5]{V2017}, the following key properties of the condition number are straightforward to prove.

\begin{proposition}
The condition number is scale and orthogonally invariant: for every $\tensor{A} = \Phi(\tuple{p}) \in \Sec{r}{\Var{S}}$, all $\alpha \in \R_0$, and all $Q = Q_1 \otimes \cdots \otimes Q_d$ with $Q_k$ an $m_k \times m_k$ orthogonal matrix, it holds that
$
\kappa( (p_1,\ldots,p_r) ) = \kappa( (\alpha Q p_1, \ldots, \alpha Q p_r) ).
$
\end{proposition}
\begin{proof}
Invariance with respect to $\alpha$ was shown in \refprop{scale_invariance}. Orthogonal invariance follows by noting that $Q_k Q_i^k$ is an orthogonal basis of the orthogonal complement of $Q_k \sten{a}{i}{k}$ so that an orthogonal basis of $\Tang{Q p_i}{\Var{S}}$ is $Q U_i$ with $U_i$ as in \refeqn{eqn_explicit_tg_space}. The result follows by the orthogonal invariance of singular values.
\end{proof}

Recall from \cite[Section 5.2]{V2017} that a CPD
\(
 \tensor{A} = \sum_{i=1}^r \sten{a}{i}{1} \otimes \cdots \otimes \sten{a}{i}{d}
\)
is called weak $3$-orthogonal
if for every $1 \le i < j \le d$ there exist $1 \le k_1 < k_2 < k_3 \le d$, depending on $i$ and~$j$, such that there is orthogonality in the corresponding factors:
\(\sten{a}{i}{k_1} \perp \sten{a}{j}{k_1}\), \(\sten{a}{i}{k_2} \perp \sten{a}{j}{k_2}\), and \(\sten{a}{i}{k_2} \perp \sten{a}{j}{k_3}\). This class includes orthogonally decomposable tensors \cite{Kolda2001,Zhang2001}.
It turns out that weak $3$-orthogonal tensors are very well conditioned.

\begin{proposition}\label{weak_3}
The condition number of a weak $3$-ortho\-go\-nal CPD is $1$.
\end{proposition}
\begin{proof}
 One verifies that $\left[\begin{smallmatrix} U_1 & \cdots & U_r \end{smallmatrix}\right]$ is a matrix with orthonormal columns due to the weak $3$-orthogonality. The result follows from the fact that a matrix with orthonormal columns has all singular values equal to $1$.
\end{proof}

\begin{corollary}
The condition number of a rank-$1$ tensor is $1$.
\end{corollary}

\subsubsection{Relation to the norm-balanced condition number}
One can ask how the proposed condition number relates to the one that was recently proposed in \cite{V2017}. The \emph{norm-balanced} condition number of a decomposition $\tuple{p} = (p_1,\ldots,p_r)$, where the $p_i$'s are as above, is defined as $\widetilde{\kappa}(\tuple{p}) := ( \varsigma_n( \left[ \begin{smallmatrix} \widetilde{U}_1 & \widetilde{U}_2 & \cdots & \widetilde{U}_r \end{smallmatrix} \right] ) )^{-1}$, where
\begin{align*}
 \widetilde{U}_i :=
\mu_i^{1-1/d}
\begin{bmatrix}
 I_{m_1} \otimes \sten{a}{i}{2} \otimes \cdots \otimes \sten{a}{i}{d} &
\cdots &
\sten{a}{i}{1} \otimes \cdots \otimes \sten{a}{i}{d-1} \otimes I_{m_d}
\end{bmatrix}.
\end{align*}
The condition number $\widetilde{\kappa}$ measures the sensitivity {of $\{ \mu_i^{1/d} \sten{a}{i}{k} \in \R^{m_k} \}_{i,k}$ with respect to perturbations of the tensor $\tensor{A} = \Phi(\tuple{p}) \in \Sec{r}{\Var{S}}$, whereas the proposed condition number $\kappa$ measures the sensitivity of the rank-$1$ tensors $(p_1, \ldots, p_r)$ in $\Var{S}^{\times r}$.}
The appropriate choice of condition number depends on the specific application.

The two condition numbers $\kappa$ and $\widetilde{\kappa}$ are related as follows. Let
\[
f: \left(\R^{m_1}_0\times\cdots\times\R^{m_d}_0\right)^{\times r} \to \R^N,\;\,
\tuple{a}=(\sten{a}{i}{1},\ldots, \sten{a}{i}{d})_{i=1,\ldots,r} \mapsto \sum_{i=1}^r \sten{a}{i}{1} \otimes \cdots \otimes \sten{a}{i}{d}.
\]
From \cite[Theorem 1]{V2017} we know that for a \emph{norm-balanced} $\tuple{a}$, i.e., $\|\sten{a}{i}{1}\|=\cdots=\|\sten{a}{i}{d}\|$, the condition number is $\widetilde{\kappa}(\tuple{p})= (\varsigma_{n} ( \deriv{f}{\tuple{a}} ) )^{-1}$, where $\tuple{p} = (p_1,\ldots,p_r)$ with $p_i = \sten{a}{i}{1}\otimes\cdots\otimes\sten{a}{i}{d}$, and $n = r \cdot \dim \Var{S}$ as before. We have the following diagram:
\begin{equation}\label{diagram}
\xymatrixcolsep{4pc}
\xymatrix{
\left(\R^{m_1}_0\times\cdots\times\R^{m_d}_0\right)^{\times r} \ar[d]^{\sigma^{\times r}} \ar[dr]^f   &\\
\left(\R^{m_1}\otimes\cdots\otimes\R^{m_d}\right)^{\times r} \supset \Var{S}^{\times r}  \ar[r]_{\hspace{2cm}\Phi} & \Sec{r}{\Var{S}}}
\end{equation}
Here $\sigma^{\times r}$ denotes the $r$-fold product of the Segre map $\sigma: (\vect{a}^1,\ldots,\vect{a}^d) \mapsto \vect{a}^1\otimes \cdots \otimes \vect{a}^d$.
Since $f=\Phi\circ \sigma^{\times r}$, we have $\deriv{f}{\tuple{a}} = \deriv{\Phi}{\tuple{p}} \circ \deriv{\sigma^{\times r}}{\tuple{a}}$. Note that $\deriv{\sigma^{\times r}}{\tuple{a}} \in \R^{n \times r(m_1 + \cdots + m_d)}$ is of full row rank because $\Var{S}$ is a smooth manifold. If $\deriv{\Phi}{\tuple{p}} \in \R^{N \times n}$ is of maximal column rank then $(\deriv{f}{\tuple{a}})^\dagger = (\deriv{\sigma^{\times r}}{\tuple{a}})^\dagger (\deriv{\Phi}{\tuple{p}})^\dagger$, by \cite[Section 5.7, Fact 16]{Hogben}. Hence,
\begin{equation}\label{eqn_two_cns_relation}
\widetilde{\kappa}(\tuple{p})
:= \| (\deriv{f}{\tuple{a}})^\dagger \|_2
\leq \| (\deriv{\sigma^{\times r}}{\tuple{a}})^\dagger \|_2 \| (\deriv{\Phi}{\tuple{p}})^\dagger \|_2
= \bigl( \varsigma_{n}(\deriv{\sigma^{\times r}}{\tuple{a}}) \bigr)^{-1} \cdot \kappa(\tuple{p}).
\end{equation}
Observe that if $\deriv{\Phi}{\tuple{p}}$ is not injective then $\widetilde{\kappa}(\tuple{p}) = \kappa(\tuple{p}) = \infty$.

\begin{proposition}
Let $\tuple{a} = (\sten{a}{i}{1},\ldots,\sten{a}{i}{d})_{i=1}^r$ be norm-balanced. Then, inequality \refeqn{eqn_two_cns_relation} is sharp for weak $3$-orthogonal CPDs $\tuple{p}=\sigma^{\times r}(\tuple{a})$.
\end{proposition}
\begin{proof}
By \refprop{weak_3}, $\kappa(\tuple{p}) = 1$, whereas \cite[Proposition 19]{V2017} states that $\widetilde{\kappa}(\tuple{p})$ is $(\mu_{\min})^{1/d - 1}$, where $\mu_{\min}$ is the smallest value in $\{ \mu_i := \|p_i\| \}_{i=1}^r$. One verifies that the matrix of $\deriv{\sigma^{\times r}}{\tuple{a}}$ with respect to the correct bases is $\operatorname{diag}(\widetilde{U}_1,\ldots,\widetilde{U}_r)$. Consequently, $\varsigma_{n}(\deriv{\sigma^r}{\vect{a}}) = (\mu_{\min})^{1 - 1/d}$.
\end{proof}

There is an interesting difference between the two condition numbers when measuring distances to $\Sec{r-1}{\Var{S}}$.
Assume that the the norms of the rank-$1$ tensors $p_i$ are sorted: $\mu_1 \ge \cdots \ge \mu_r > 0$.
If $\mu_r$ is small then a small perturbation applied to $\mu_r p_r$ sends it to $0$, hereby sending the perturbed tensor to $\sigma_{r-1}$. The condition number at the perturbed point is $\widetilde{\kappa} = \varsigma_n( \left[\begin{smallmatrix} \widetilde{U}_1 & \cdots & \widetilde{U}_{r-1} & 0 \end{smallmatrix}\right] )^{-1} = \infty$. It follows from the continuity of singular values that as $\mu_r \to 0$, the condition number $\widetilde{\kappa} \to \infty$ so that for small $\mu_r$ the absolute norm-balanced condition number of \cite{V2017} is automatically large.
By contrast, $\kappa$ is insensitive to the norm of the points $p_i$, as was shown in \refprop{scale_invariance}, hence $\mu_r \approx 0$ does not automatically imply a large condition number.
It also entails that if a perturbation of magnitude less than $\mu_r$ is applied to $p=p_1+\cdots+p_r$ in the direction of $p_r$, then the condition number remains constant. This type of perturbation does not decrease the projection distance from $(\Tang{p_1}{\Var{S}},\ldots,\Tang{p_r}{\Var{S}})$ to $\GrSigma$, which is exactly the inverse of $\kappa(\tuple{p})$ according to \refthm{thm_boundary_points_cn}.
The only way to perturb $p$ towards $\Sec{r-1}{\Var{S}}$ so that the condition number increases is to change the angles between the $p_i$'s, which causes the angles between $\Tang{p_i}{\Var{S}}$ to shift, hereby changing the projection distance to $\GrSigma$.

\subsection{The Waring decomposition}
A symmetric tensor $\tensor{A} \in \R^{m \times \cdots \times m}$ is a tensor whose entries are invariant under a permutation of indices: $a_{i_1,\ldots,i_d} = a_{\pi_1, \ldots, \pi_d}$ where $\pi$ is any permutation of $\{i_1,\ldots, i_d\}$.
Such tensors can be decomposed as
\[
 \tensor{A} = \sum_{i=1}^r \sten{a}{i}{\otimes d} = \sum_{i=1}^r \sten{a}{i}{} \otimes \cdots \otimes \sten{a}{i}{} \quad\text{with } \sten{a}{i}{} \in \R^m;
\]
if $r$ is minimal, this expression is called a Waring decomposition. The summands in generic Waring decompositions are uniquely determined \cite{COV2017}. This type of decompositions occurs often in latent variable models \cite{AGHKT2014}, where the individual rank-$1$ symmetric tensors $\vect{a}_i$ correspond to the quantities of interest.

The set of rank-$1$ symmetric tensors is the affine cone over a smooth projective variety called the Veronese variety \cite{Landsberg2012}; hence, it is a manifold. It can be constructed as
the image of the surjective map
\[
 \Psi : \R_0 \times \mathbb{S}(\R^{m}) \to \Var{V}, \quad (\mu_i, \vect{a}_i) \mapsto \mu \vect{a}_i^{\otimes d} = \mu_i \vect{a}_i \otimes \cdots \otimes \vect{a}_i.
\]
The join set $\Var{J} = \JOIN(\Var{V},\ldots,\Var{V})$ is then the set of all symmetric tensors that have a Waring decomposition of rank bounded by $r$.

The condition number of the Waring decomposition problem (WDP) is determined analogously to the previous example. The derivative of $\Psi$ is
\begin{align*}
 \deriv{\Psi}{(\mu_i,\vect{a}_i)} : \Tang{\mu_i}{\R_0} \times \Tang{\vect{a}_i}{\mathbb{S}(\R^m)} &\to \Tang{\Psi(\mu_i,\vect{a}_i)}{\Var{V}} \\
 (\dot{\mu}_i, \dot{\vect{a}}_i ) &\mapsto \dot{\mu}_i \vect{a}_i^{\otimes d} + \mu_i \dot{\vect{a}}_i \otimes \vect{a}_i^{\otimes d-1} + \cdots + \vect{a}_i^{\otimes d-1} \otimes \dot{\vect{a}}_i.
\end{align*}
Let $Q_i$ be an $m \times (m-1)$ matrix whose columns form an orthonormal basis of $\Tang{\vect{a}_i}{\mathbb{S}(\R^m)}$. Let $p_i = \mu_i \sten{a}{i}{\otimes d}$. Then, we see that an orthonormal basis of $\Tang{p_i}{\Var{V}}$ is given by the columns of the $m^d \times m$ matrix\footnote{In fact, $\Var{V} \subset S^d \R^m \simeq \R^{\binom{d+m-1}{d}}$, so it is possible to reduce the size of this matrix to $\binom{d+m-1}{d} \times m$. For simplicity we ignore this optimization.}
\[
 V_i :=
 \begin{bmatrix}
  \sten{a}{i}{\otimes d} & \frac{1}{d}(Q_i \otimes \sten{a}{i}{\otimes d-1} + \cdots + \sten{a}{i}{\otimes d-1} \otimes Q_i)
 \end{bmatrix};
\]
the fact that $V_i^T V_i = I_m$ can be verified readily. Then it follows from \refthm{thm_condition_number} that the condition number of the WDP at $\tuple{p} = (p_1,\ldots,p_r)$ is
\(
 \kappa(\tuple{p}) = ( \varsigma_{n} ( \left[\begin{smallmatrix} V_1 & \cdots & V_r \end{smallmatrix}\right] ) )^{-1},
\)
where $n = \dim \Var{V}^{\times r} = r m$.

The following results are obtained by arguments very similar to those employed in the corresponding results in \refsec{sec_tensor_rank_decomposition}.

\begin{proposition}
The condition number is scale and orthogonally invariant: for every $\tensor{A} = \Phi(\tuple{p}) \in \Sec{r}{\Var{V}}$, all $\alpha \in \R_0$, and all $Q = Q' \otimes \cdots \otimes Q'$ with $Q'$ an $m \times m$ orthogonal matrix, it holds that $
\kappa( (p_1,\ldots,p_r) ) = \kappa( (\alpha Q p_1, \ldots, \alpha Q p_r) ).
$
\end{proposition}

Recall that a symmetric \emph{odeco} tensor $\tensor{A}$ can be expressed as $\sum_{i=1}^r \sten{a}{i}{\otimes d}$ with $\sten{a}{i}{} \perp \sten{a}{j}{}$ for all $i\ne j$ and $r \le n$. Since Kruskal's lemma \cite{JS2004} applies, the set of symmetric rank-$1$ summands in the decomposition $\tuple{p} \in \Var{V} \times \cdots \times \Var{V}$ is uniquely determined. We can thus call $\kappa(\tuple{p})$ \textit{the} condition number of a symmetric odeco tensor $\Phi(\tuple{p})$.

\begin{proposition}
The condition number of a symmetric odeco tensor is $1$.
\end{proposition}

\begin{corollary}
The condition number of a symmetric rank-$1$ tensor is $1$.
\end{corollary}

\section{Numerical experiments}\label{sec_numerical_experiments}
We illustrate the theoretical results in the setting of CPDs. We performed several experiments in Matlab R2016b, employing some features of Tensorlab v3.0 \cite{Tensorlab}.
The code implementing the computation of the condition number presented in \refsec{sec_tensor_rank_decomposition} is included in the ancillary files of the arXiv version of this paper.

In this section, low-rank tensors are specified by \emph{factor matrices} $A_k = [\sten{a}{i}{k}]_{i=1}^r$ with $\sten{a}{i}{k} \in \R^{m_k}$; the associated tensor is $\tensor{A} := \ldbracket A_1,\ldots,A_d \rdbracket := \sum_{i=1}^r \sten{a}{i}{1} \otimes \cdots \otimes \sten{a}{i}{d}$. By $X \sim N(0,1)$ we mean that all elements of $X \in \R^{m \times n}$ are standard normally distributed with zero mean and unit standard deviation.

\subsection{Estimating typical forward errors}\label{sec_experiment2}
While the bound in \refeqn{eqn_rule_of_thumb}, namely $\| \tuple{p} - \Phi_{\tuple{p}}^{-1}(\vect{w})\| \lesssim \kappa(\tuple{p}) \cdot \| \Phi(\tuple{p}) - \vect{w}\|$,  is asymptotically sharp, it may be questioned insofar this is a \emph{typical} result. Perhaps the estimate obtained from the condition number is overly pessimistic compared to the errors that one encounters in practice?

For investigating the above question, we conducted the following experiment.
A model for generating potentially ill-conditioned CPD problems is the tensor in $\R^{m_1 \times \cdots \times m_d}$ whose $k$th factor matrix is
\begin{align}\label{eq_basic_model}
 A_k(s) =  C_k (2^{-as} \cdot I_r + X_k Y_k^T) \quad\text{with } C_k, X_k, Y_k \sim N(0,1),
\end{align}
and where $a > 0$, $C_k \in \R^{m_k \times r}$ and $X_k, Y_k \in \R^{r \times r_k}$ with $r_k < m_k$. These tensors are of rank $r$ and maximal multilinear rank with probability $1$, and as $s \to \infty$ the factor matrix $A_k(s)$ is tending to the rank-$r_k$ matrix $C_k X_k Y_k^T$. The tensor $\tensor{A}(s) = \ldbracket A_1(s),\ldots,A_d(s) \rdbracket$ will tend to a tensor whose multilinear rank is at most $(r_1,\ldots,r_d)$ as $s \to \infty$. Recall that a tensor of multilinear rank $(r_1,\ldots,r_d)$ in $\R^{m_1 \times \cdots \times m_d}$ is $r$-identifiable if and only if the $r_1 \times \cdots \times r_d$ core array in its higher-order singular value decomposition (HOSVD) \cite{Lathauwer2008} is $r$-identifiable \cite[Theorem 4.1]{COV2014}. In particular if $r$ is strictly larger than the generic complex rank of $\R^{r_1 \times \cdots \times r_d}$ \cite{BT2015}, then $\tensor{A}(s)$ tends to a tensor with infinitely many decompositions with probability $1$. {We expect that the condition number of $\tensor{A}(s)$ increases with $s$ because of \refprop{prop_infdecomp_infcond}.}

We take $(m_1,m_2,m_3,m_4) = (6,5,4,4)$, $(r_1,r_2,r_3,r_4) = (1,2,3,4)$, $r=6$, $a = \frac{1}{5}$, and $s = 1,2,\ldots,50$ as an instance of the above model. Then, the tensor
\[
\tensor{A}(s) = \ldbracket \|A_1(s)\|_F^{-1} A_1(s), \ldots, \|A_d(s)\|_F^{-1} A_d(s) \rdbracket = \ldbracket B_1(s), \ldots, B_d(s) \rdbracket
\]
was computed. For every $s$, we tried to recover the CPD of $\tensor{A}(s)$ using the \texttt{cpd\_nls} optimization method from Tensorlab. It was halted either if the value of the objective function $\frac{1}{2}\| \Phi(\tuple{p}(s)) - \tensor{A}(s)\|^2 \le 10^{-14}$ or after $1000$ iterations. As initial factor matrices from whence \texttt{cpd\_nls} started we chose
\(
 \left\{ B_1(s) + \tau R_1, \ldots, B_d(s) + \tau R_d \right\},
\)
where $R_k \sim N(0,1)$ and $\tau = 5 \cdot 10^{-4}$. Let $\widetilde{\tensor{A}}(s) = \ldbracket \widetilde{B}_1(s), \ldots, \widetilde{B}_d(s) \rdbracket$ denote the CPD computed by \texttt{cpd\_nls}. Then, the following data was recorded: the backward error $\| \tensor{A}(s) - \widetilde{\tensor{A}}(s) \|_F$, the forward error $\| B_1(s) \odot \cdots \odot B_d(s) - \widetilde{B}_1(s) \odot \cdots \odot \widetilde{B}_d(s) \|_F$, and the condition number $\kappa(s)$ of the problem at the computed CPD $\widetilde{B}(s)$. For every $s=1,2,\ldots,50$, we took $250$ independent samples from model \refeqn{eq_basic_model} with the aforementioned choice of parameters. Out of the $12500$ experiments, $16$ were discarded because the \texttt{cpd\_nls} method did not converge within $1000$ iterations.

\begin{figure}[tb]
 \caption{Distribution of the scaling factor $\frac{\| B_1(s) \odot \cdots \odot B_d(s) - \widetilde{B}_1(s) \odot \cdots \odot \widetilde{B}_d(s) \|_F}{\kappa(s) \cdot \| \tensor{A}(s) - \widetilde{\tensor{A}}(s) \|_F} \lesssim 1$ for $250$ random samples of model \refeqn{eq_basic_model}.}
 \label{fig_deciles}
 \begin{center}
\begingroup
  \makeatletter
  \providecommand\color[2][]{%
    \GenericError{(gnuplot) \space\space\space\@spaces}{%
      Package color not loaded in conjunction with
      terminal option `colourtext'%
    }{See the gnuplot documentation for explanation.%
    }{Either use 'blacktext' in gnuplot or load the package
      color.sty in LaTeX.}%
    \renewcommand\color[2][]{}%
  }%
  \providecommand\includegraphics[2][]{%
    \GenericError{(gnuplot) \space\space\space\@spaces}{%
      Package graphicx or graphics not loaded%
    }{See the gnuplot documentation for explanation.%
    }{The gnuplot epslatex terminal needs graphicx.sty or graphics.sty.}%
    \renewcommand\includegraphics[2][]{}%
  }%
  \providecommand\rotatebox[2]{#2}%
  \@ifundefined{ifGPcolor}{%
    \newif\ifGPcolor
    \GPcolortrue
  }{}%
  \@ifundefined{ifGPblacktext}{%
    \newif\ifGPblacktext
    \GPblacktexttrue
  }{}%
  \let\gplgaddtomacro\g@addto@macro
  \gdef\gplbacktext{}%
  \gdef\gplfronttext{}%
  \makeatother
  \ifGPblacktext
    \def\colorrgb#1{}%
    \def\colorgray#1{}%
  \else
    \ifGPcolor
      \def\colorrgb#1{\color[rgb]{#1}}%
      \def\colorgray#1{\color[gray]{#1}}%
      \expandafter\def\csname LTw\endcsname{\color{white}}%
      \expandafter\def\csname LTb\endcsname{\color{black}}%
      \expandafter\def\csname LTa\endcsname{\color{black}}%
      \expandafter\def\csname LT0\endcsname{\color[rgb]{1,0,0}}%
      \expandafter\def\csname LT1\endcsname{\color[rgb]{0,1,0}}%
      \expandafter\def\csname LT2\endcsname{\color[rgb]{0,0,1}}%
      \expandafter\def\csname LT3\endcsname{\color[rgb]{1,0,1}}%
      \expandafter\def\csname LT4\endcsname{\color[rgb]{0,1,1}}%
      \expandafter\def\csname LT5\endcsname{\color[rgb]{1,1,0}}%
      \expandafter\def\csname LT6\endcsname{\color[rgb]{0,0,0}}%
      \expandafter\def\csname LT7\endcsname{\color[rgb]{1,0.3,0}}%
      \expandafter\def\csname LT8\endcsname{\color[rgb]{0.5,0.5,0.5}}%
    \else
      \def\colorrgb#1{\color{black}}%
      \def\colorgray#1{\color[gray]{#1}}%
      \expandafter\def\csname LTw\endcsname{\color{white}}%
      \expandafter\def\csname LTb\endcsname{\color{black}}%
      \expandafter\def\csname LTa\endcsname{\color{black}}%
      \expandafter\def\csname LT0\endcsname{\color{black}}%
      \expandafter\def\csname LT1\endcsname{\color{black}}%
      \expandafter\def\csname LT2\endcsname{\color{black}}%
      \expandafter\def\csname LT3\endcsname{\color{black}}%
      \expandafter\def\csname LT4\endcsname{\color{black}}%
      \expandafter\def\csname LT5\endcsname{\color{black}}%
      \expandafter\def\csname LT6\endcsname{\color{black}}%
      \expandafter\def\csname LT7\endcsname{\color{black}}%
      \expandafter\def\csname LT8\endcsname{\color{black}}%
    \fi
  \fi
  \setlength{\unitlength}{0.0500bp}%
\scalebox{.82}{
  \begin{picture}(8640.00,4400.00)%
    \gplgaddtomacro\gplbacktext{%
      \csname LTb\endcsname%
      \put(990,704){\makebox(0,0)[r]{\strut{}$10^{-1}$}}%
      \csname LTb\endcsname%
      \put(990,1037){\makebox(0,0)[r]{\strut{}$10^{-0.9}$}}%
      \csname LTb\endcsname%
      \put(990,1369){\makebox(0,0)[r]{\strut{}$10^{-0.8}$}}%
      \csname LTb\endcsname%
      \put(990,1702){\makebox(0,0)[r]{\strut{}$10^{-0.7}$}}%
      \csname LTb\endcsname%
      \put(990,2035){\makebox(0,0)[r]{\strut{}$10^{-0.6}$}}%
      \csname LTb\endcsname%
      \put(990,2368){\makebox(0,0)[r]{\strut{}$10^{-0.5}$}}%
      \csname LTb\endcsname%
      \put(990,2700){\makebox(0,0)[r]{\strut{}$10^{-0.4}$}}%
      \csname LTb\endcsname%
      \put(990,3033){\makebox(0,0)[r]{\strut{}$10^{-0.3}$}}%
      \csname LTb\endcsname%
      \put(990,3366){\makebox(0,0)[r]{\strut{}$10^{-0.2}$}}%
      \csname LTb\endcsname%
      \put(990,3698){\makebox(0,0)[r]{\strut{}$10^{-0.1}$}}%
      \csname LTb\endcsname%
      \put(990,4031){\makebox(0,0)[r]{\strut{}$10^{0}$}}%
      \csname LTb\endcsname%
      \put(1122,484){\makebox(0,0){\strut{} 0}}%
      \csname LTb\endcsname%
      \put(1617,484){\makebox(0,0){\strut{} 5}}%
      \csname LTb\endcsname%
      \put(2111,484){\makebox(0,0){\strut{} 10}}%
      \csname LTb\endcsname%
      \put(2606,484){\makebox(0,0){\strut{} 15}}%
      \csname LTb\endcsname%
      \put(3100,484){\makebox(0,0){\strut{} 20}}%
      \csname LTb\endcsname%
      \put(3595,484){\makebox(0,0){\strut{} 25}}%
      \csname LTb\endcsname%
      \put(4089,484){\makebox(0,0){\strut{} 30}}%
      \csname LTb\endcsname%
      \put(4584,484){\makebox(0,0){\strut{} 35}}%
      \csname LTb\endcsname%
      \put(5078,484){\makebox(0,0){\strut{} 40}}%
      \csname LTb\endcsname%
      \put(5573,484){\makebox(0,0){\strut{} 45}}%
      \csname LTb\endcsname%
      \put(6067,484){\makebox(0,0){\strut{} 50}}%
      \put(3594,154){\makebox(0,0){\strut{}$s$}}%
    }%
    \gplgaddtomacro\gplfronttext{%
      \csname LTb\endcsname%
      \put(7651,3921){\makebox(0,0)[r]{\strut{}1st Decile}}%
      \csname LTb\endcsname%
      \put(7651,3701){\makebox(0,0)[r]{\strut{}2nd Decile}}%
      \csname LTb\endcsname%
      \put(7651,3481){\makebox(0,0)[r]{\strut{}3rd Decile}}%
      \csname LTb\endcsname%
      \put(7651,3261){\makebox(0,0)[r]{\strut{}4th Decile}}%
      \csname LTb\endcsname%
      \put(7651,3041){\makebox(0,0)[r]{\strut{}5th Decile}}%
      \csname LTb\endcsname%
      \put(7651,2821){\makebox(0,0)[r]{\strut{}6th Decile}}%
      \csname LTb\endcsname%
      \put(7651,2601){\makebox(0,0)[r]{\strut{}7th Decile}}%
      \csname LTb\endcsname%
      \put(7651,2381){\makebox(0,0)[r]{\strut{}8th Decile}}%
      \csname LTb\endcsname%
      \put(7651,2161){\makebox(0,0)[r]{\strut{}9th Decile}}%
    }%
    \gplbacktext
    \put(0,0){\includegraphics{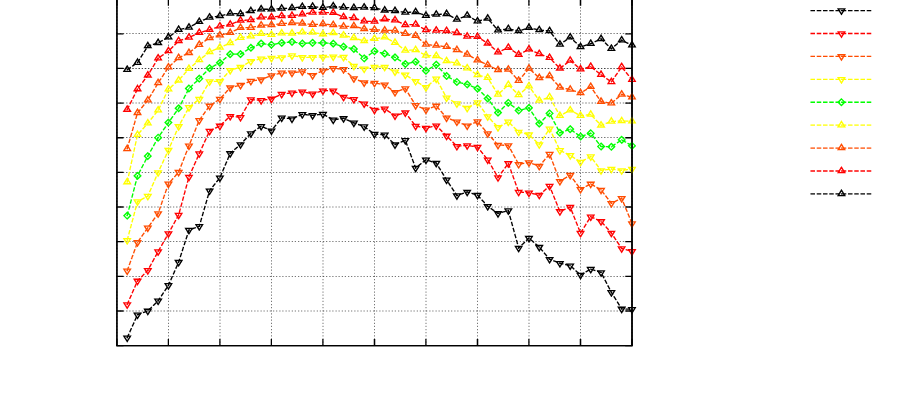}}%
    \gplfronttext
  \end{picture}%
}
\endgroup
 \end{center}
\end{figure}

\begin{table}[tb]
\caption{The quartiles of the observed condition numbers in the experiment described in \refsec{sec_experiment2}.}
\label{tab_cns}
\[\footnotesize
\begin{array}{lcccccc}
\toprule
s & 1 & 10 & 20 & 30 & 40 & 50 \\
\midrule
\text{1st quartile} 	&  1.27 \cdot10^{1} & 5.40 \cdot10^{1} & 2.60 \cdot10^{2} & 1.22 \cdot10^{3} & 5.03 \cdot10^{3} & 2.00 \cdot10^{4} \\
\text{median} 		&  2.04 \cdot10^{1} & 1.02 \cdot10^{2} & 4.88 \cdot10^{2} & 2.09 \cdot10^{3} & 8.55 \cdot10^{3} & 3.44 \cdot10^{4} \\
\text{3rd quartile} 	&  3.27 \cdot10^{1} & 1.77 \cdot10^{2} & 8.68 \cdot10^{2} & 3.65 \cdot10^{3} & 1.50 \cdot10^{4} & 6.03 \cdot10^{4} \\
\bottomrule
\end{array}
\]
\end{table}

The obtained results are summarized in \reffig{fig_deciles} and \reftab{tab_cns}. The quartiles of the observed condition numbers are shown in \reftab{tab_cns}; note that they indeed increase with $s$. Figure \ref{fig_deciles} plots the deciles of the forward error divided by the asymptotic upper bound on the forward error, i.e., the backward error multiplied with the condition number. The figure thus shows for every $s$ that in at least $90\%$ of the cases the asymptotic upper bound on the forward error was at most $10$ times larger than the actual forward error. The curve of the $9$th decile even shows that in $10\%$ of the cases for every $s$ the observed forward error was at least $10^{-0.2} \approx 0.6$ times the asymptotic upper bound on the forward error. This provides solid evidence that the condition number multiplied with the backward error \emph{is not an overly pessimistic estimate of the forward error}. In fact, in this particular model \emph{the forward error of the computed CPD is only fractionally less than the asymptotic upper bound.}
It follows from \reftab{tab_cns} that the backward error, which was approximately $\sqrt{2} \cdot 10^{-7}$ in every experiment, multiplied with the condition number more accurately estimates the forward error than the backward error---the latter is still the dominant criterion for evaluating the quality of CPDs in absence of the forward error.

\subsection{Open boundary points}\label{sec_obp_experiments}
In the next experiments the condition number is computed for a sequence of rank-$r$ tensors that converges to an open boundary point, i.e., a tensor of rank at least $r+1$. According to \refthm{thm_boundary_points_cn}, the condition number should increase without bound along this sequence.

\begin{figure}[bt]
\caption{The condition number for the sequences \refeqn{eqn_sequence1} and \refeqn{eqn_sequence2}.}
\label{fig_convergence_obp}
\begin{center}
\begingroup
  \makeatletter
  \providecommand\color[2][]{%
    \GenericError{(gnuplot) \space\space\space\@spaces}{%
      Package color not loaded in conjunction with
      terminal option `colourtext'%
    }{See the gnuplot documentation for explanation.%
    }{Either use 'blacktext' in gnuplot or load the package
      color.sty in LaTeX.}%
    \renewcommand\color[2][]{}%
  }%
  \providecommand\includegraphics[2][]{%
    \GenericError{(gnuplot) \space\space\space\@spaces}{%
      Package graphicx or graphics not loaded%
    }{See the gnuplot documentation for explanation.%
    }{The gnuplot epslatex terminal needs graphicx.sty or graphics.sty.}%
    \renewcommand\includegraphics[2][]{}%
  }%
  \providecommand\rotatebox[2]{#2}%
  \@ifundefined{ifGPcolor}{%
    \newif\ifGPcolor
    \GPcolortrue
  }{}%
  \@ifundefined{ifGPblacktext}{%
    \newif\ifGPblacktext
    \GPblacktexttrue
  }{}%
  \let\gplgaddtomacro\g@addto@macro
  \gdef\gplbacktext{}%
  \gdef\gplfronttext{}%
  \makeatother
  \ifGPblacktext
    \def\colorrgb#1{}%
    \def\colorgray#1{}%
  \else
    \ifGPcolor
      \def\colorrgb#1{\color[rgb]{#1}}%
      \def\colorgray#1{\color[gray]{#1}}%
      \expandafter\def\csname LTw\endcsname{\color{white}}%
      \expandafter\def\csname LTb\endcsname{\color{black}}%
      \expandafter\def\csname LTa\endcsname{\color{black}}%
      \expandafter\def\csname LT0\endcsname{\color[rgb]{1,0,0}}%
      \expandafter\def\csname LT1\endcsname{\color[rgb]{0,1,0}}%
      \expandafter\def\csname LT2\endcsname{\color[rgb]{0,0,1}}%
      \expandafter\def\csname LT3\endcsname{\color[rgb]{1,0,1}}%
      \expandafter\def\csname LT4\endcsname{\color[rgb]{0,1,1}}%
      \expandafter\def\csname LT5\endcsname{\color[rgb]{1,1,0}}%
      \expandafter\def\csname LT6\endcsname{\color[rgb]{0,0,0}}%
      \expandafter\def\csname LT7\endcsname{\color[rgb]{1,0.3,0}}%
      \expandafter\def\csname LT8\endcsname{\color[rgb]{0.5,0.5,0.5}}%
    \else
      \def\colorrgb#1{\color{black}}%
      \def\colorgray#1{\color[gray]{#1}}%
      \expandafter\def\csname LTw\endcsname{\color{white}}%
      \expandafter\def\csname LTb\endcsname{\color{black}}%
      \expandafter\def\csname LTa\endcsname{\color{black}}%
      \expandafter\def\csname LT0\endcsname{\color{black}}%
      \expandafter\def\csname LT1\endcsname{\color{black}}%
      \expandafter\def\csname LT2\endcsname{\color{black}}%
      \expandafter\def\csname LT3\endcsname{\color{black}}%
      \expandafter\def\csname LT4\endcsname{\color{black}}%
      \expandafter\def\csname LT5\endcsname{\color{black}}%
      \expandafter\def\csname LT6\endcsname{\color{black}}%
      \expandafter\def\csname LT7\endcsname{\color{black}}%
      \expandafter\def\csname LT8\endcsname{\color{black}}%
    \fi
  \fi
  \setlength{\unitlength}{0.0500bp}%
  \scalebox{.8}{
  \begin{picture}(8640.00,3528.00)%
    \gplgaddtomacro\gplbacktext{%
      \csname LTb\endcsname%
      \put(528,846){\makebox(0,0)[r]{\strut{}$10^{2}$}}%
      \csname LTb\endcsname%
      \put(528,1223){\makebox(0,0)[r]{\strut{}$10^{4}$}}%
      \csname LTb\endcsname%
      \put(528,1599){\makebox(0,0)[r]{\strut{}$10^{6}$}}%
      \csname LTb\endcsname%
      \put(528,1975){\makebox(0,0)[r]{\strut{}$10^{8}$}}%
      \csname LTb\endcsname%
      \put(528,2352){\makebox(0,0)[r]{\strut{}$10^{10}$}}%
      \csname LTb\endcsname%
      \put(528,2728){\makebox(0,0)[r]{\strut{}$10^{12}$}}%
      \csname LTb\endcsname%
      \put(528,3104){\makebox(0,0)[r]{\strut{}$10^{14}$}}%
      \csname LTb\endcsname%
      \put(528,3481){\makebox(0,0)[r]{\strut{}$10^{16}$}}%
      \csname LTb\endcsname%
      \put(1447,484){\makebox(0,0){\strut{} 10}}%
      \csname LTb\endcsname%
      \put(2321,484){\makebox(0,0){\strut{} 20}}%
      \csname LTb\endcsname%
      \put(3195,484){\makebox(0,0){\strut{} 30}}%
      \csname LTb\endcsname%
      \put(4069,484){\makebox(0,0){\strut{} 40}}%
      \csname LTb\endcsname%
      \put(4943,484){\makebox(0,0){\strut{} 50}}%
      \csname LTb\endcsname%
      \put(5818,484){\makebox(0,0){\strut{} 60}}%
      \csname LTb\endcsname%
      \put(6692,484){\makebox(0,0){\strut{} 70}}%
      \csname LTb\endcsname%
      \put(7566,484){\makebox(0,0){\strut{} 80}}%
      \csname LTb\endcsname%
      \put(8440,484){\makebox(0,0){\strut{} 90}}%
      \put(4550,154){\makebox(0,0){\strut{}$s$}}%
      \put(4550,3395){\makebox(0,0){\strut{}}}%
    }%
    \gplgaddtomacro\gplfronttext{%
      \csname LTb\endcsname%
      \put(3168,3332){\makebox(0,0)[r]{\strut{}Sequence \cref{eqn_sequence1}}}%
      \csname LTb\endcsname%
      \put(3168,3112){\makebox(0,0)[r]{\strut{}Sequence \cref{eqn_sequence2}}}%
    }%
    \gplbacktext
    \put(0,0){\includegraphics{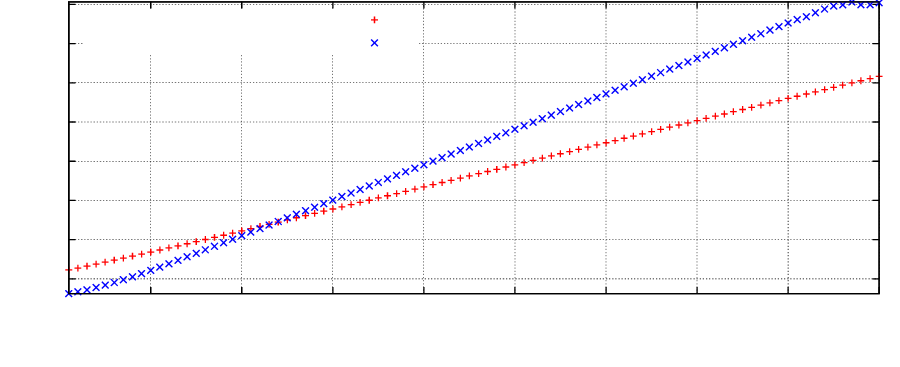}}%
    \gplfronttext
  \end{picture}}
\endgroup
\end{center}
\end{figure}

For brevity we consider the examples of Paatero \cite[Section 7]{Paatero2000}, and de Silva and Lim \cite[Theorem 1.1]{dSL2008}. The first sequence consists of the rank-$3$ tensors
\begin{align}
\nonumber\tensor{A}(s) =
 &-2^{\frac{3s}{16}} (\sten{a}{1}{1} + \sten{a}{2}{1}) \otimes \sten{a}{1}{2} \otimes \sten{a}{1}{3}  +2^{\frac{3s}{16}} \sten{a}{2}{1} \otimes (\sten{a}{1}{2} + 2^{-\frac{3s}{16}-1} \sten{a}{3}{2}) \otimes (\sten{a}{1}{3} + 2^{-\frac{3s}{16}-1} \sten{a}{3}{3})
  \\
 &+ 2^{\frac{3s}{16}} (\sten{a}{1}{1} + 2^{-\frac{3s}{16}-1} \sten{a}{3}{1}) \otimes (\sten{a}{1}{2} + 2^{-\frac{3s}{16}-1} \sten{a}{2}{2}) \otimes (\sten{a}{1}{3} + 2^{-\frac{3s}{16}-1} \sten{a}{2}{3}),\label{eqn_sequence1}
\end{align}
where $A_k = [\sten{a}{i}{k}]_{i=1}^3 \in \R^{m_k \times 3}$ and $A_k \sim N(0,1)$. We took $(m_1,m_2,m_3) = (5,4,3)$ so that $\tensor{A}(s)$ has a unique rank-$3$ decomposition for all $s$ with probability $1$. As $s \to \infty$, $\tensor{A}(s)$ tends to a rank-$5$ tensor, which even has a positive-dimensional family of rank-$5$ CPDs \cite[Section 7]{Paatero2000}.
The sequence of rank-$2$ tensors from \cite{dSL2008} is
\begin{align}\label{eqn_sequence2}
 \tensor{B}(s) = 2^{\frac{s}{5}} (\sten{b}{1}{1} + 2^{-\frac{s}{5}} \sten{b}{2}{1}) \otimes (\sten{b}{1}{2} + 2^{-\frac{s}{5}} \sten{b}{2}{2}) \otimes (\sten{b}{1}{3} + 2^{-\frac{s}{5}} \sten{b}{2}{3}) - 2^{\frac{s}{5}} \sten{b}{1}{1}\otimes\sten{b}{1}{2}\otimes\sten{b}{1}{3},
\end{align}
where $B_k = [\sten{b}{i}{k}]_{i=1}^2 \in \R^{m_k \times 2}$ and $B_k \sim N(0,1)$. We took $(m_1,m_2,m_3) = (5,3,2)$ so that $\tensor{B}(s)$ has a unique rank-$2$ decomposition for all $s$ with probability $1$. The limit for $s\to\infty$ of this sequence is a rank-$3$ tensor. In both cases, we randomly sampled the factor matrices $A_k$ and $B_k$ and then varied $s$. For every $s=1,2,\ldots,90$, the condition number was computed.

The results are summarized in \reffig{fig_convergence_obp}. Both graphs show greatly increasing condition numbers as $s$ increases, consistent with \refthm{thm_boundary_points_cn}. Comparing \reffig{fig_convergence_obp} with Figure 2 from \cite{V2017} reveals qualitatively similar behavior of the condition numbers in \cite[Theorem 1]{V2017} and \refthm{thm_condition_number}.

From plotting the condition numbers of the sequence \cref{eqn_sequence1} and \cref{eqn_sequence2} we are convinced that for increasing $s$ the respective decompositions become ill-posed. But \emph{how} they become ill-posed is not so evident. To get a clearer picture we consider another experiment. In \cite[Equation (21)]{Paatero2000} Paatero gives the following parametrized family of $2\times 2\times 2$ tensors
\begin{equation}\label{eqn_sequence3}
\tensor{C}(s,t) = (\sten{e}{1}{}\otimes \sten{e}{2}{}+\sten{e}{2}{}\otimes \sten{e}{1}{}+ s\,\sten{e}{2}{}\otimes \sten{e}{2}{})\otimes\sten{e}{1}{} + (30\,\sten{e}{1}{}\otimes \sten{e}{1}{}-t\,\sten{e}{2}{}\otimes \sten{e}{2}{})\otimes\sten{e}{2}{},
\end{equation}
where $\sten{e}{1}{},\sten{e}{2}{},\sten{e}{3}{}$ are the standard basis vectors in $\R^2$.
Paatero shows that for $\tfrac{15}{2} s^2 < t$ the tensor $\tensor{C}(s,t)$ has rank $3$, whereas for $\tfrac{15}{2} s^2 >t$ the tensor $\tensor{C}(s,t)$ has rank $2$. Note that $\dim \R^{2\times 2\times 2} = 8$ and $8= 2\dim \Var S$. The generic fibers of $\Phi:\Var S^{\times 3} \to \R^{2\times 2\times 2}$ are positive-dimensional and, consequently, the condition number of a decomposition of a generic rank-3 tensor is ill-posed. The parabola $\tfrac{15}{2} s^2=t$ consists of open boundary points. Hence, the condition number must increase when heading towards this parabola from the side $\tfrac{15}{2} s^2 >t$.

\begin{figure}[bt]
\caption{The base-10 logarithm $\log_{10}$ of the condition number for Paatero's model \cref{eqn_sequence3}. The grid consists of $100\times 100$ points equally spaced in $[-0.4,0.4]\times [-0.4,0.4]$.}
\label{fig_paatero}
\begin{center}
\includegraphics[width=0.9\textwidth]{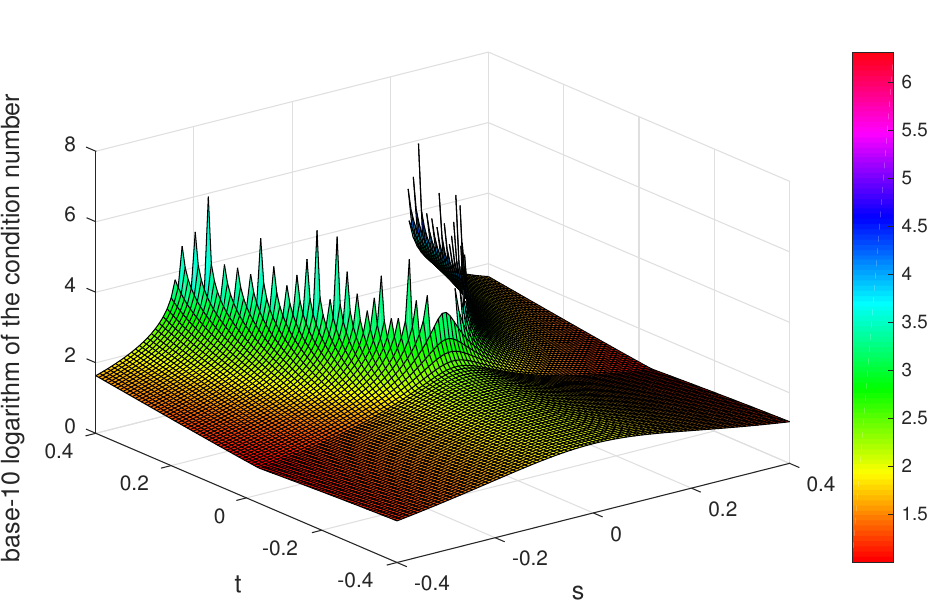}
\end{center}
\end{figure}

The experiment we performed is as follows. On a $100 \times 100$ uniformly spaced grid in $[-0.4,0.4]\times [-0.4,0.4]$, we computed CPDs for $\tensor{C}(s,t)$ using the direct decomposition algorithm \texttt{cpd\_gevd} from Tensorlab v3.0 \cite{Tensorlab}. Then we computed the condition number of the decompositions. The outcome of this experiment can be seen in \cref{fig_paatero}. Note in particular spikes close to the parabola $\tfrac{15}{2} s^2 =t$.

\section{Conclusions}\label{sec_conclusions}
A local condition number for the join decomposition problem on $\Var{J} = \JOIN(\Var{M}_1,\ldots,\Var{M}_r)$ was presented in this paper. We gave both an easily computable spectral characterization as well as a characterization as an inverse projection distance to a locus of ill-posed problems in an auxiliary product Grassmannian. We believe that one main application of the condition number lies in the infamous rule of thumb of numerical analysis in \refeqn{eqn_rule_of_thumb}.

An immediate consequence of \refthm{thm_condition_number} is that a join decomposition $\tuple{p}$ has $\kappa(\tuple{p}) = \infty$ when $\Phi(\tuple{p})$ locally has a smooth positive-dimensional family of alternative decompositions $\Var{E}$ so that $\Phi(\tuple{p}) = \Phi(\Var{E})$; this was stated as \refprop{prop_infdecomp_infcond}.

The JDP is ill-posed at open boundary points. Provided that the join set satisfies some technical conditions, \refthm{thm_boundary_points_cn} proved that open boundary points are completely surrounded by ill-conditioned JDPs: on a sequence of join decompositions tending to an open boundary point the condition number tends to infinity.

Two examples of join decompositions in the context of tensors were investigated more closely: the CPD and the Waring decomposition. For these examples, a closed expression of $U$ in \refthm{thm_condition_number} was given, which is easily amenable to a computer implementation. We additionally presented closed expressions of the condition number for certain subclasses of tensors.

We hope that the proposed condition number will find application in the analysis of join decompositions and their JDPs. In fact, we have recently shown in \cite{BV2017_2} that the condition number naturally appears in the convergence analysis of certain Riemannian optimization methods for solving \refeqn{eqn_prototype_optim_problem}. Moreover, in \cite{BV2018} we demonstrated that great computational savings can be realized for solving the canonical tensor rank approximation problem compared to state-of-the-art algorithms in Tensorlab v3.0 by exploiting information about the condition number of the CPD in Riemannian optimization algorithms.

\subsection*{Acknowledgements}
We thank Peter B\"urgisser for all his helpful comments and advice. In particular, he pointed out to us some references used in \refsec{sec_def_condition_number}.
We want to thank the organizers of the \emph{Workshop on Tensor Decompositions and Applications 2016} in Leuven, Belgium, where the authors first met. Finally, we thank the editor D. Kressner and the anonymous referees for their remarks which improved this paper.


\bibliographystyle{siamplain}
\bibliography{BV}
\end{document}